%% file: simpl_conv.tex
\documentclass[11pt,reqno]{amsart}
\usepackage{amssymb}
\usepackage{pgfplots}
\usepackage{csvsimple}
\usepackage{siunitx}
\usepackage{hyperref}
\usepackage{graphicx,color}
\usepackage[normalem]{ulem}
\usepackage{amsthm,amsxtra,amscd,mathtools}

\include{sb_macros-mod}
\providecommand{\abstmp}[2]{{#1\lvert{#2}#1\rvert}}
\providecommand{\bigabs}[1]{\abstmp{\big}{#1}}
\providecommand{\Bigabs}[1]{\abstmp{\Big}{#1}}
\providecommand{\biggabs}[1]{\abstmp{\bigg}{#1}}

\numberwithin{equation}{section}
\newcommand{\eps}{\varepsilon}					   
\newcommand{\abs}[1]{\lvert#1\rvert} 		   
\newcommand\norm[1]{\left\lVert#1\right\rVert} 
	
\newcommand{\fdg}{{\,\big|\,}}

\newcommand{\ue}{\prescript{\eps\hspace{-.25mm}}{}u}
\newcommand{\bue}{\prescript{\eps\hspace{-.25mm}}{}{\bar u}}
\newcommand{\hue}{\prescript{\eps\hspace{-.25mm}}{}{\hat u}}
\newcommand{\Ee}{\prescript{\eps\hspace{-.45mm}}{}E}
\newcommand{\Fe}{\prescript{\eps\hspace{-.45mm}}{}F}
\newcommand{\divo}{\operatorname{div}}
\newcommand{\h}{d}
\newcommand{\gH}{D}

\newtheorem{ass}[equation]{Assumption}
\newtheorem{lemma}[equation]{Lemma}
\newtheorem{proposition}[equation]{Proposition}
\newtheorem{theorem}[equation]{Theorem}

\newtheorem{remark}[equation]{Remark}

\begin{document}
\title[Discrete approximations of nonlinear
  parabolic equations]{Convergence of fully discrete implicit
 and semi-implicit approximations of nonlinear
  parabolic equations}
\author{S\"oren~Bartels}
\address{Abteilung f\"ur Angewandte Mathematik,  
Albert-Ludwigs-Universit\"at Freiburg, Hermann-Herder-Str. 10, 
79104 Freiburg i.Br., Germany}
\email{bartels@mathematik.uni-freiburg.de}
\author{Michael R{\r u}{\v z}i{\v c}ka}
\address{Abteilung f\"ur Angewandte Mathematik,  
Albert-Ludwigs-Universit\"at Freiburg, Ernst-Zermelo-Str. 1, 
79104 Freiburg i.Br., Germany}
\email{rose@mathematik.uni-freiburg.de}

\date{\today}
\begin{abstract}
The article addresses the convergence of implicit and semi-implicit, fully
discrete approximations of a class of nonlinear parabolic 
evolution problems. Such schemes are popular in
the numerical solution of evolutions defined with the $p$-Laplace 
operator since the latter lead to linear systems of equations
in the time steps. The semi-implicit treatment of the operator
requires introducing a regularization parameter that has to 
be suitably related to other discretization parameters. To
avoid restrictive, unpractical conditions, a careful convergence
analysis has to be carried out. The arguments presented in this
article show that convergence holds under a moderate condition
that relates the step size to the regularization parameter but
which is independent of the spatial resolution. 
\end{abstract}
\maketitle

\section{Introduction}
It has recently been shown in the article~\cite{BDN} that the 
semi-implicit time stepping scheme for the $p$-Laplace gradient flow defined with 
an initial function~$u^0$ via the recursion
\begin{align}\label{eq:p-lap}
d_\tau u^k = \diver \frac{\nabla u^k}{|\nabla u^{k-1}|_\eps^{2-p}} 
\end{align}
with the regularized norm $|a|_\eps = (|a|^2+\eps^2)^{1/2}$ and the 
backward difference quotient operator $d_\tau = (u^k-u^{k-1})/\tau$ 
is unconditionally energy stable. Specifically, this means that the estimate 
\begin{equation*}
E_{p,\eps}[u^L] + \tau \sum_{k=1}^L \|d_\tau u^k\|_{L^2(\O)}^2 
+ \frac{\tau^2}{2} \sum_{k=1}^L 
\int_\O \frac{|\nabla d_\tau u^k|^2}{|\nabla u^{k-1}|_\eps^{2-p}} \dv{x}
\le E_{p,\eps}[u^0]
\end{equation*}
holds for all $\tau,\eps>0$ and $1\le p\le 2$ 
and all $L\ge 1$ with the regularized $p$-Dirichlet energy
\[
E_{p,\eps}[u] = \frac1p \int_\O |\nabla u|_\eps^p \dv{x}.
\]
The energy estimate follows from testing~\eqref{eq:p-lap} with
$d_\tau u^k$ using special identities from finite difference calculus
and certain monotonicity properties of the $p$-Laplace operator.  An
error analysis for a generic spatial discretization with
mesh-size~$h>0$ of the scheme leads to an upper bound for the
approximation error in $L^\infty(0,T;L^2(\Omega))$ involving the term
\[
\tau^{1/2} (h\eps)^{(p-2)/2}.
\]
To deduce a convergence rate for the error the restrictive condition
$\tau = o((h\eps)^{2-p})$ has to be satisfied. The aim of this note
is to show that the sequence of piecewise constant interpolants 
of the iterates $(u^k_h)_{k=0,\dots,K}$, ${h>0}$,  (weakly) converges to the 
solution of the continuous flow under the less restrictive condition
$\tau = O(\eps^{2-p})$ independently of the mesh-size $h>0$ and
even for a larger class of operators also
including lower order contributions. 

To explain our ideas we interpret the iterates $(u^k)_{k=0,\dots,K}$ of the
semi-implicit scheme as iterates of an implicit, unregularized scheme 
with discrepancy terms on the right-hand sides, i.e., with
the $L^2$ inner product $(\cdot,\cdot)$ we have
\begin{equation}\label{eq:impl_repr}
(d_\tau u^k,v) + \big(|\nabla u^k|^{p-2} \nabla u^k, \nabla v\big)
= (\cD^k, \nabla v).
\end{equation}
Using the operator 
\[
S_\eps(a) = \frac{a}{|a|_\eps^{2-p}} 
\]
we rewrite the discrepancy terms as 
\[\begin{split}
\cD^k &= \big[|\nabla u^k|^{p-2} - |\nabla u^{k-1}|_\eps^{p-2}\big] \nabla u^k \\
&= \big(S_0(\nabla u^k) - S_\eps(\nabla u^k)\big)
 + \big(S_\eps(\nabla u^k) - |\nabla u^{k-1}|^{p-2}_\eps \nabla u^k\big) = E^k + F^k.
\end{split}\]
The first term on the right-hand side is controlled using the uniform
convergence property 
\[
|S_\eps(a)-S_0(a)|\le (2-p) \eps^{p-1}
\]
which follows from the mean value estimate 
$\big| |a|^{2-p} - |a|_\eps^{2-p}\big| \le (2-p)|a|^{1-p} \eps$ for $a\neq 0$.
Therefore, we have
\[
(E^k, \nabla v) \le (2-p)\eps^{p-1} \|\nabla v\|_{L^1(\O)}.
\]
To bound the second term on the right-hand side we use the estimate 
\[
|S_\eps(a)-S_\eps(b)|\le c_p |a-b| \big(\eps^2 + |a|^2 + |b|^2\big)^{(p-2)/2},
\]
cf.~\cite{DiEbRu07}, which leads to 
\[\begin{split}
(F^k,\nabla v)
&= \int_\O  \big(S_\eps(\nabla u^k) - S_\eps(\nabla u^{k-1}) + |\nabla u^{k-1}|^{p-2}_\eps \nabla [u^{k-1}-u^k]\big) \cdot \nabla v \dv{x}  \\
&\le \frac{(c_p+1)^2 \tau\a_\eps}{2} \int_\O \frac{|\nabla d_\tau u^k|^2}{|\nabla u^{k-1}|_\eps^{2-p}} \dv{x}
+ \frac{\tau \eps^{p-2}}{2 \a_\eps} \int_\O |\nabla v|^2 \dv{x}. 
\end{split}\]
Letting $\overline{\cD}$ be the piecewise constant interpolation of $\cD^k$ and
integrating the estimate for $\cD^k$ over $(0,T)$ we thus obtain with the energy bound that 
\[\begin{split}
\int_0^T (\overline{\cD},\nabla v) \dv{t} & \le 
(2-p) \eps^{p-1} \int_0^T \|\nabla v\|_{L^1(\O)}\dv{t}  \\
& + c_p^2 \a_\eps \tau^2 \sum_{k=1}^K \int_\O \frac{|\nabla d_\tau u^k|^2}{|\nabla u^{k-1}|_\eps^{2-p}} \dv{x}
+   \frac{\tau \eps^{p-2}}{2 \a_\eps} \int_0^T \|\nabla v\|_{L^2(\O)}^2  \dv{t},
\end{split}\]
where $\a_\eps>0$ is arbitrary. 
Choosing, e.g., $\a_\eps = (\tau\eps^{p-2})^{1/2}$, and requiring $\tau = o(\eps^{2-p})$ 
we find that the discrepancy term converges to zero whenever $v\in L^2(0,T;W^{1,2}_0(\O))$. 
If an implicit discretization of the $p$-Laplace gradient flow is known to converge
to the exact solution then it follows that the iterates of the semi-implicit 
scheme~\eqref{eq:p-lap} also converge to this object. 

Surprisingly, a rigorous convergence analysis for the fully discrete,
implicit 
scheme for the $p$-Laplace evolution does not
seem to be available in the literature. Classical references such
as~\cite{GGZ}, \cite{zei-IIB}, \cite{showalter}, and~\cite{Roub} consider 
semi-discrete schemes,
i.e., either Galerkin methods corresponding to a spatial
discretization or Rothe methods realizing implicit time stepping
schemes. Full discretizations lead to additional analytical
difficulties as, e.g., the schemes only provide limited control on the
time derivatives. To avoid the construction of a stable projection
operator a generalized Aubin--Lions lemma has been established
in~\cite{Roub}. An alternative to this is based on the
Hirano--Landes lemma, which ensures the convergence in the nonlinear
operator provided an energy estimate can be established and a
generalized condition (M) can be verified based on the approximate
equations and the properties of the nonlinear operator (cf.~\cite
{br-hirano} for previous versions of this approach). Another approach
to establishing convergence of solutions can be based on the framework
of subdifferential flows but this limits the analysis to convex
energies and excludes other nonlinearities. In order to verify the
generalized condition (M) we require in addition to the energy
estimate stated above also bounds resulting from
testing the scheme~\eqref{eq:p-lap} by~$u^k$.

Various error estimates are available for numerical approximations of $p$-Laplace
evolutions and related equations, see, e.g.,~\cite{BarLiu94,rulla,NoSaVe00,FeOePr05,DiEbRu07}. 
These are typically valid under certain regularity conditions,
impose relations between discretization parameters, or consider only 
implicit time-stepping schemes. Here, we are interested
in establishing convergence of the approximations obtained with the practical
semi-implicit scheme~\eqref{eq:p-lap} under moderate conditions on the relation between the
step-size parameter~$\tau$ and the regularization parameter~$\eps$. Therefore,
we cannot resort to those results when we affiliate the convergence to the
convergence of an implicit scheme with discrepancy terms. 

To establish the convergence of the iterates $(u^k)_{k=0,\dots,K}$ of
the semi-implicit scheme~\eqref{eq:p-lap}, even when a spatial
discretization is carried out, we first consider the corresponding
implicit scheme and prove that appropriate interpolants weakly
accumulate at an exact solution. This result is the consequence of a
general convergence result for a fully discrete implicit
approximation proved in an abstract framework for evolution equations
with pseudo-monotone operators. Typical examples of such operators are
sums of a monotone and a compact operator.
Only moderate assumptions will be made on the data and on the
discretizations. A technical condition on the finite element spaces
requires sequences of finite element spaces to be nested as the
mesh-size tends to zero. 

The outline of this article is as follows. In Subsection~\ref{sec:prelim}
we define a class of energy densities that lead to admissible
operators to which our arguments apply. In
Section~\ref{sec:fully_impl} we derive a convergence result for
approximations obtained with a fully discrete implicit scheme for
general evolution equations with pseudo-monotone operators. This
serves as a guideline to show that the approximations obtained with a
semi-implicit, practical scheme generalizing~\eqref{eq:p-lap}
for a large class of monotone evolutions including lower order
contributions converges to a solution. 

Throughout this article we let $\O\subset \R^d$, $d\ge 2$, be a bounded Lipschitz
domain and use standard notation for Lebesgue and Sobolev spaces. Most
results apply to bounded open sets $\O$ but in view of numerical 
discretizations we consider the slightly stronger condition. We denote the
inner product in $L^2(\O)$ by $(\cdot,\cdot)$ and the duality pairing
of a Banach space $V$ with its dual $V'$ which often extends the $L^2$ 
inner product by $\langle \cdot, \cdot \rangle_V$.


\subsection{Properties of the nonlinear operator}\label{sec:prelim}
For a given convex function $\vphi: \R_{\ge 0}\to \R_{\ge 0}$ we
consider energy functionals $E_\vphi: L^1(\O) \to \R\cup\{+\infty\}$
defined via
\[
E_\vphi[u] = \int_\O \vphi(|\nabla u|) \dv{x}.
\]
We denote by $W^{1,\vphi}(\O)$ the set of weakly differentiable 
functions $u\in L^1(\O)$ for which we have $E_\vphi[u]<\infty$. 
We make the following assumptions on the energy density $\vphi$
which define a class of sub-quadratic Orlicz functions.
\begin{ass}[Energy density]\label{ass:energy}
Let $\vphi\colon \R_{\ge 0} \to \R_{\ge 0}$ belong to $C^0(\R_{\ge 0}) \cap
C^1(\R_{>0})$. We assume that 
\begin{itemize}
\item[\textnormal{(C1)}] $r\mapsto \vphi(r)$ is convex with $\vphi(0)=0$.
\item[\textnormal{(C2)}] $r\mapsto \vphi'(r)/r$ is positive and
  nonincreasing. 
\end{itemize}
\end{ass}
Sometimes we additionally make the following assumption. 
\begin{ass}[N-function]\label{ass:N}
Let $\vphi\colon \R_{\ge 0} \to \R_{\ge 0}$ belong to $C^1(\R_{\ge 0}) \cap
C^2(\R_{>0})$. We assume that 
\begin{itemize}
\item[\textnormal{(C3)}] The function $\vphi $ is 
convex and positive on $(0,\infty)$,
satisfies $\vphi(0)=0$, and $\lim_{s\to 0} \vphi(s)/s = 0$ and 
$\lim_{s\to \infty} \vphi(s)/s=\infty$; moreover $\vphi$ and its convex
conjugate $\vphi^*$ satisfy $\vphi(2s)\lesssim \vphi(s)$ and
$\vphi^*(2r)\lesssim \vphi(r)$ 
for all $r,s \in \R_{\ge 0}$. Finally we assume that there exist 
constants $\kappa_0 \in (0,1]$, $\kappa{_1} >0$ such that for all $r
\in \R_{>0}$
\begin{align*}
  \kappa_0 \vphi'(r) \le r\vphi''(r) \le \kappa_1\vphi '(r)\,.
\end{align*}
\end{itemize}
\end{ass}

For a given N-function $\vphi$ we define the {\rm shifted N-functions}
$\{\vphi_\alpha \}_{\alpha  \ge 0}$, cf.~\cite{die-ett,die-kreu,dr-nafsa}, for
$t\geq0$ by
\begin{align}
  \label{eq:phi_shifted}
  \vphi_\alpha (t):= \int _0^t \vphi_\alpha '(s)\, ds\qquad\text{with }\quad
  \vphi'_\alpha (t):=\vphi'(\alpha +t)\frac {t}{\alpha +t}.
\end{align}
If $\vphi$ satisfies the conditions (C1), (C2) and (C3), than the
family of shifted N-functions $\{\vphi _\alpha\}_{\alpha\ge 0}$ also
satisfies conditions (C1), (C2) and (C3).  The family of shifted
N-functions $\{\vphi_\alpha\}_{\alpha\ge 0}$ induces operators
$A_\alpha\colon \R^d\to \R^d$ with potential $\vphi_\alpha$ via
\begin{align}
  \label{eq:op-a}
  A_\alpha(a):= \frac {\vphi_\alpha'(\abs{a})}{\abs{a}}a\,.
\end{align}
One easily checks the following relations (cf.~\cite{die-ett,die-kreu,dr-nafsa}).
\begin{lemma}\label{la:mon_props_b}
If $\vphi$ satisfies~(C3), then the following statements are valid: \\
(i) For all $a,b\in \R^d$ and all $\alpha\ge 0$ we have with
constants independent of $\alpha$  
\begin{align}
  \big(A_\alpha(a)-A_\alpha(b)\big)\cdot (a-b) & \eqsim  (\vphi_\alpha)_{|a|} (|a-b|), \label{eq:op_A_mon-a} \\
  \big|A_\alpha(a)-A_\alpha(b)\big| &\eqsim (\vphi_\alpha)_{|a|}'(|a-b|), \label{eq:op_A_cont-a}
\end{align}
and
\begin{equation}\label{eq:sim_phi-a}
  (\vphi_\alpha)_{|a|} (|a-b|) \eqsim \frac{\vphi_\alpha'(|a|+|b|)}{|a|+|b|} |a-b|^2.
\end{equation}
(ii) For all $\d>0$ there exists $c_\d$ such that for all
$\a,r,s \ge 0$ we have
\begin{equation}\label{eq:young_phi}
  \vphi_\a'(r) s \le c_\d \vphi_\a(r) + \d \vphi_\a(s).
\end{equation}
(iii) For all $\d$ there exists $c_\d$ such that for all $a, b \in
\R^d$, and all $r\ge 0$
\begin{align}
  \label{eq:shift}
  \begin{split}
    \vphi_{|a|} (r) &\le c_\d\,\vphi_{\abs{b}}(r) + \d\,
    \vphi_{\abs{a}} (|a-b|),
    \\
    (\vphi_{|a|})^* (r) &\le c_\d\,(\vphi_{\abs{b}})^*(r) + \d\,
    \vphi_{\abs{a}} (|a-b|).
  \end{split}
\end{align}
Moreover, we have $\vphi_{\abs{a}} (|a-b|)\eqsim \vphi_{\abs{b}} (|a-b|)$
\end{lemma}

We need some further properties related to the function $\vphi$. In
the same way as in \cite{BDN} one can prove the following inequality. 
\begin{lemma}\label{la:orlicz_stab}
Under condition $(C2)$ we have for all $a,b\in \R^d$ and all $\eps \ge
0$ that
\[
  \frac{\vphi_\eps'(|a|)}{|a|} b \cdot (b-a) \ge \vphi_\eps(|b|)- \vphi_\eps(|a|) 
  + \frac12 \frac{\vphi_\eps'(|a|)}{|a|} |b-a|^2. 
\]
\end{lemma}

To handle the difference between the implicit scheme and the
semi-implicit scheme, the following estimate is useful.
\begin{lemma}\label{lem:est}
If $\vphi$ satisfies~(C2), (C3), then we have for all $a,b \in \R^d$,
$\eps \ge 0$
\begin{align*}
  \biggabs{ \Big (\frac{\vphi'_\eps (\abs{a})}{\abs{a}}
  -\frac{\vphi'_\eps (\abs{b})}{\abs{b}} \Big) a} \lesssim 
  \frac{\vphi'_\eps (\abs{b})}{\abs{b}} \abs{a-b}\,.  
\end{align*}
\end{lemma}
\begin{proof}
  We have
  \begin{align*}
    \biggabs{ \Big (\frac{\vphi'_\eps (\abs{a})}{\abs{a}}
    -\frac{\vphi'_\eps (\abs{b})}{\abs{b}} \Big) a} 
    &=  \Bigabs{ A_\eps (a) - A_\eps (b) + \frac{\vphi'_\eps
      (\abs{b})}{\abs{b}} (b- a)} 
    \\
    &\lesssim (\vphi_\eps')_{\abs{a}} (\abs{a-b})  +\frac{\vphi'_\eps
      (\abs{b})}{\abs{b}} \abs{b- a}
    \\
    &\lesssim \frac{\vphi_\eps' (\abs{b} +\abs{b-a})}{\abs{b}
      +\abs{b-a}}\abs{b-a}  +\frac{\vphi'_\eps
      (\abs{b})}{\abs{b}} \abs{b- a}
    \\
    &\le 2\, \frac{\vphi'_\eps
      (\abs{b})}{\abs{b}} \abs{b- a}\,,
  \end{align*}
  where we used that $\abs{b} +\abs{b-a} \eqsim \abs{b} +\abs{a}$ and
  condition (C2).
\end{proof}
We have a uniform convergence property for the operators $A_\eps$.
\begin{lemma}\label{lem:eps}
If $\vphi$ satisfies~(C2), (C3), then we have for all $a\in \R^d$,
$\eps \ge 0$
\begin{align*}
  \bigabs{ A_\eps (a) -A(a)} \le (1-\kappa_0)\vphi'(\eps)\,.
\end{align*}  
\end{lemma}
\begin{proof}
  For $a=0$ or $\eps =0$ the estimate is clear. Thus, we assume in the following
  $\abs{a}>0$ and $ \eps>0$. Setting $f(t):=\frac {t}{\vphi'(t)}$, $t>0$, we see
  from (C2) that $f$ is nondecreasing.  Moreover, from (C3) we obtain
  that  $0\le f'(s)= \frac 1{\vphi'(s)}\Big (1- \frac {s \vphi''(s)}{\vphi
    '(s)}\Big )\le \frac {1-\kappa_0}{\vphi'(s)}$. From the mean value
  theorem we get for all $t>0$, $\eps > 0$
  \begin{align*}
    \bigabs{f(t+\eps )-f(t)} 
    &= \eps\,  {f'(\zeta)} \le \eps\, \frac
      {1-\kappa_0}{\vphi'(\zeta)} \le \eps\,
      \frac  {1-\kappa_0}{\vphi'(t)} ,
  \end{align*}
  where we used that $\zeta \in (t,t+\eps)$ and that $\vphi '$ is
  increasing. Thus we get
\begin{align*}
  \bigabs{ A_\eps (a) -A(a)} 
  &= \biggabs{\frac{\vphi'(\eps +\abs{a})}{\eps +\abs{a}} -
    \frac{\vphi'(\abs{a})}{\abs{a}}   }\, \abs{a}
  \\
  &= \biggabs{\frac{f(\abs{a})-f(\eps +\abs{a})}{f(\abs{a})\,f(\eps
    +\abs{a})} }\, \abs{a} 
  \\
  &\le  \eps\, \frac {1-\kappa_0}{\vphi'(\abs{a})}
    {\frac{\vphi'(\abs{a} ) \, \vphi'(\eps +\abs{a})}{\abs{a}\,(\eps 
    +\abs{a})} }\, \abs{a} 
  \\
  & \le \eps\, (1-\kappa_0) \frac{ \vphi'(\eps +\abs{a})}{\eps 
    +\abs{a}} \le (1-\kappa_0)\vphi'(\eps)\,,
\end{align*}  
where we used also (C2). 
\end{proof}
Prototypical examples for functions $\vphi$ satisfying the conditions (C1),
(C2) and (C3) are N-functions with $(p,\delta)$-structure.
We say that an N-function $\varphi \in C^1(\R_{\ge 0}) \cap
C^2(\R_{>0}) $ has $(p,\delta)$-structure, with $p\in (1,\infty)$ and
$\delta\ge 0$, if
\begin{equation}
\begin{alignedat}{2}\label{eq:ass}
  \varphi(t) &\eqsim (\delta+t)^{p-2}t^2\,,\qquad &&\textrm{uniformly in $t\ge
            0$}\,,
  \\
  \varphi''(t) &\eqsim (\delta+t)^{p-2}\,,\qquad &&\textrm{uniformly in $t>
            0$}
  \,.
\end{alignedat}
\end{equation}
The constants in these equivalences and $p$ are called characteristics
of $\varphi$. A detailed discussion of N-functions with
$(p,\delta)$-structure can e.g.~be found in~\cite{r-cetraro}. Using \eqref{eq:ass} and the change of shift
\eqref{eq:shift} we easily see that for all $\eps, \delta \ge 0$ we
have uniformly in $t \ge 0$
\begin{align}
  \label{eq:equi}
  \varphi_\eps (t) +\eps ^p + \delta^p\eqsim   t^p +\eps ^p + \delta^p
\end{align}
with constants only depending on $p$.



\section{Convergence of an implicit scheme}\label{sec:fully_impl}

In this section we study abstract evolution equations with
pseudo-monotone operators. Concrete realizations of this situation
will be discussed in the next section.

Let $V$ be a Banach space. An operator $B\colon V\to V^*$ is said to
be monotone if $\langle Bx-By,x-y\rangle_V \geq 0$ for all
$x,y \in V$.  The operator ${B \colon V\to V^*}$ is said to be
pseudo-monotone if $x_n \rightharpoonup x$ in V and
$\limsup_{n\rightarrow \infty} \langle Bx_n,x_n-x \rangle_V\leq 0$
implies
\begin{align*}
  \langle Bx,x-y \rangle_V \leq \liminf\limits_{n\rightarrow \infty} \,
  \langle Bx_n, x_n -y \rangle_V \qquad \text{for all } y \text{ in }
  V. 
\end{align*}  

Let $V$ be a separable, reflexive Banach space and $H$ a Hilbert
space. If the embedding $V\hookrightarrow H$ is dense, we call
$(V,H,V^\ast)$ a Gelfand-Triple. Using the Riesz representation
theorem we obtain
$V\hookrightarrow H \cong H^\ast \hookrightarrow V^\ast$ where both
embeddings are dense. In this situation there holds
$(x,y)_H=\langle x,y\rangle _V=\langle y,x\rangle _V$ for all
$x,y \in V$.  
We say that a function $u \in L^p(0,T;V)$ possesses a generalized
derivative in $L^{p'}(0,T;V^\ast)$, where
$\frac{1}{p}+\frac{1}{p'}=1$, if there is a function
$w \in L^{p'}(0,T;V^\ast)$ such that
\begin{align*}
  \int_0^T \left(u(t),v   \right)_H\phi'(t) \, dt = -
  \int_0^T\left\langle w(t) ,v\right\rangle_V\phi(t) \, dt 
\end{align*}
for all $v \in V$ and all $\phi \in C_0^\infty((0,T))$. If such a
function $w$ exists, it is unique and we set $\frac{du}{dt} :=w$. We
define the Bochner--Sobolev space
\begin{align*}
W_p^1(0,T;V,H):= \big\lbrace u\in L^p(0,T;V) \fdg \frac{du}{dt} \in L^{p'}(0,T;V^\ast) \big\rbrace.
\end{align*}
With the norm 
\begin{align*}
\norm{u}_{W_p^1(0,T;V,H)}:= \norm{u}_{L^p(0,T;V)} + \norm{\frac{du}{dt}}_{L^{p'}(0,T;V^\ast) }
\end{align*}
this space is a reflexive Banach space. Moreover, we have that
$W_p^1(0,T;V,H)$ embeds continuously into $C(0,T;H)$ and the following
integration by parts formula
\begin{align*}
  \left( u(t) ,v(t)\right)_H- \left( u(s) ,v(s)\right)_H = \int_s^t
  \left\langle \frac{du}{dt}(\tau),v(\tau)\right\rangle_V 
  +\left\langle \frac{dv}{dt}(\tau),u(\tau)\right\rangle_V \, d\tau
\end{align*}
holds for any $u,v \in W_p^1(0,T;V,H)$ and arbitrary $0\leq s,t\leq T$
(cf.~\cite[Proposition 23.23]{zei-IIA}).

We study the following evolution equation with a
pseudo-monotone operator $B$:
\begin{align}\label{eq:pm}
  \begin{aligned}
    \frac{du}{dt}(t) +Bu(t)&=f(t) &&\text{in }V^\ast \text{ for a.e.} \ t \in [0,T],\\
    u(0)&=u^0 &&\text{in}\ H.
  \end{aligned}
\end{align}

To establish the existence of solutions we make the following
assumptions on the operator $B$.
\begin{ass}[Operator]\label{voroperatorfamilie}
  Let $(V,H,V^*)$ be a Gelfand triple and let $B\colon V\to V^\ast$ be
  an operator with the following properties:
\begin{itemize}
\item[\textnormal{(A1)}] $B$ is pseudo-monotone. 
\item[\textnormal{(A2)}] There exist constants $c_1>0$,  $c_2 \ge 0$,
  $c_3 \ge 0$,  such that for all $x\in V$
\begin{align*}
  \langle Bx,x  \rangle_V \geq c_1 \norm{x}_V^{p} -c_2 \norm{x}_H^2 -c_3\,.
\end{align*}
\item[\textnormal{(A3)}] There exists $0\le q<\infty$, as well as
  constants $c_4>0$, $c_5\geq 0$ and $c_6\ge 0$, such that for all $x\in V$
\begin{align*}
  \norm{Bx}_{V^\ast}\leq c_4\norm{x}_V^{p-1}+ c_5 \norm{x}_H^q
  \norm{x}_V^{p-1} +  c_6\, .
\end{align*}
\end{itemize}
\end{ass}
Under this assumption we have (cf.~\cite[Chapters 27, 30]{zei-IIB}):
\begin{lemma}\label{induzierterOp}
  Assume that the operator $B\colon V\to V^\ast$ satisfies Assumption
  \ref{voroperatorfamilie}. Then the in\-duced operator
  $({B}u)(t):=Bu(t)$ maps the space \linebreak ${ L^p(0,T;V)\cap
  L^\infty(0,T;H)}$ into $L^{p'}(0,T;V^\ast)$ and is bounded.
\end{lemma}

Previous existence results that we are aware of are based on either a Rothe approximation
(cf.~\cite{Roub}) or a Galerkin approximation (cf.~\cite{br-hirano}).
We want to establish the existence of a solution of \eqref{eq:pm} with the
help of a convergence proof of a Rothe-Galerkin scheme. To this end we
introduce some notation. For each
$K \in \N$ we set ${\tau:=\frac TK}$, $t_k= t_k^\tau:=k \tau $, $k=0,\ldots, K$
and $I_k=I_k^\tau:=(t_{k-1},t_k]$,  ${k=1,\ldots, K}$. The backward difference
quotient operator is defined as 
$$
{d_\tau c^k:=\tau^{-1}(c^k-c^{k-1})}.
$$
For a given finite sequence $(c^k)_{k=0,\ldots, K}$ we denote by
$\bar c^\tau$ the piecewise constant interpolant and by $\hat c^\tau$
the piecewise affine interpolant, i.e.~$\hat c^\tau(t)=\big( \frac t\tau -(k-1)\big)c^k +\big(k- \frac t\tau
\big)c^{k-1}$, $\bar c ^\tau (t) =c^k $, $t \in I_k$, $\bar c^\tau (0)=\hat c ^\tau
(0)=c^0$. Note that $\frac {d\hat c^\tau(t)}{dt}= d_\tau c^k$ for all
$t \in (t_{k-1},t_k)$.

\begin{ass}[Data]\label{ass:data}
  Let $(V,H,V^*)$ be a Gelfand triple. We assume that $u^0\in H$ and
  $f \in L^{p'}(0,T;V^\ast)$. Moreover, we assume that there exists an
  increasing sequence of finite dimensional subspaces $V_M$, $M\in \N$,
  such that $\bigcup _{M\in \N} V_M$ is dense in $V$.  Finally, we
  assume that there exist $u_M^0 \in V_M$ such that $u_M^0 \to u^0$ in
  $H$, and that there exists a sequence $f_M \in C(0,T;V^\ast)$ such
  that $f_M\to f $ in $L^{p'}(0,T;V^\ast)$.
\end{ass}

For each $M\in \N$ and given $u^0_M \in V_M$ the sequence of iterates
$(u^k_M)_{k=0,\ldots, K}\subseteq V_M$ is given via the implicit scheme
\begin{align}
  \label{eq:rg}
  \langle d_\tau u_M^k,v_M\rangle _V + \langle Bu_M^k,v_M\rangle _V= \langle
  f_M(t_k),v_M\rangle_V \qquad \forall v_M \in V_M.
\end{align}

\begin{theorem}[Convergence of the implicit scheme]\label{maintheorem}
  Let Assumption \ref{voroperatorfamilie} and \ref{ass:data} be
  satisfied. Let $\bar u_n:=\bar u_{M_n}^{\tau_n}$ be a sequence of
  piecewise constant interpolants generated by iterates
  $(u^k_{M_n})_{k=0,\ldots, K_n}$, $K_n=\frac T{\tau_n}$, solving
  \eqref{eq:rg} for some sequences $M_n\to \infty$, $\tau_n\to 0$.
  Then each weak$^\ast $ accumulation point $u$ of the sequence
  $(\bar u_n)_{n\in \N}$ in the space $L^\infty(0,T;H)\cap L^p(0,T;V)$
  belongs to the space $ W_p^1(0,T;V,H)$ and is a solution of
  \eqref{eq:pm}.
\end{theorem}

The proof of this theorem is based on a generalization of Hirano's
lemma (cf.~\cite{Shi}, \cite{Roub}) using ideas from \cite{landes-87},
\cite{landes-86}. The advantage of this generalization is that it
avoids a technical assumption on the existence of suitable projections
(cf.~\cite{br-hirano}).
\begin{proposition}[Hirano, Landes]\label{Hirano}
  Let Assumption \ref{voroperatorfamilie}  be
  satisfied. Further assume that the sequence 
  $(u_n)$ is bounded in $  L^p(0,T;V)\cap L^\infty(0,T;H) $ and satisfies
  \begin{align}\label{Hiranovoraussetzung}
    \begin{aligned}
      u_n&\rightharpoonup u\quad\text{ in }L^p(0,T;V),\\
      u_n&\stackrel{*}\rightharpoonup u\quad\text{ in }L^\infty(0,T;H),\\
      u_n(t)&\rightharpoonup u(t)\quad\text{ in }H\text{ for almost all
      }t\in (0,T),\\
      \limsup\limits_{n\rightarrow \infty}\, &\langle
      {B}u_n,u_n-u\rangle_{L^p(0,T;V)} \leq 0.
    \end{aligned}
  \end{align} 
  Then for any $z \in L^p(0,T;V)$ there holds
  \begin{align}\label{Hiranoresultat}
    \langle {B}u,u -z\rangle_{L^p(0,T;V)}
    \leq\liminf\limits_{n\rightarrow \infty}\, \langle
    {B}u_n,u_n-z\rangle_{L^p(0,T;V)}. 
  \end{align}
  Moreover, ${B}u_n \rightharpoonup {B}u$ in $L^p(0,T;V)^\ast= L^{p'}(0,T;V^\ast)$.
\end{proposition}
\begin{proof}
  The proof is almost identical with the proof of \cite[Lemma
  4.2]{br-hirano}. First note that from assumptions (A2), (A3) we can
  derive for all $x \in L^p(0,T;V)\cap L^\infty(0,T;H) $ with
  $\norm{x}_{L^\infty(0,T;H) }\le K$, all $y \in L^p(0,T;V)$ and
  almost all $t \in (0,T)$
  \begin{align*}
    \langle {B}x(t),x(t) -y(t)\rangle_{L^p(0,T;V)}\ge k_1
    \norm{x(t)}_V^p- k_2 \norm{y(t)}_V^p-k_3\,,
  \end{align*}
  with positive constants $k_i$, $i=1,2,3$, depending on $K$ and
  $c_j$, $j=1,\ldots, 6$. The last inequality is exactly inequality
  (4.4) in \cite{br-hirano}, which is crucial for the proof of
  Lemma~4.2 there. Note, that assumption
  \eqref{Hiranovoraussetzung}$_3$ is not present in the formulation of
  \cite[Lemma 4.2]{br-hirano}, but it is assumed instead that $(u_n)$
  is bounded in $L^{p'}(0,T;Z^\ast)$, for a certain separable,
  reflexive Banach space $Z$ with $Z \hookrightarrow V$. This
  assumption is solely used to identify the pointwise limits
  $u_n(t) \rightharpoonup u(t)$ in $Z^\ast$ for all $t\in [0,T]$
  (cf.~\cite[equation (4.5)]{br-hirano}). This identification together
  with the embedding $V \hookrightarrow Z^\ast$ implies for a certain
  subsequence $u_{n_k}(t) \rightharpoonup u(t)$ in $V$ for almost
  all $t\in [0,T]$ (cf.~\cite[equation (4.8)]{br-hirano}). This
  argument is replaced by our assumption
  \eqref{Hiranovoraussetzung}$_3$, that also identifies the pointwise
  limits of $(u_n(t))$ in $H$.  This and the embedding
  $V\hookrightarrow H$ again yield that for a certain subsequence
  $u_{n_k}(t) \rightharpoonup u(t)$ in $V$ for almost all
  $t\in [0,T]$. After that the proof can be finished in an identical
  manner as in \cite{br-hirano}.
\end{proof}

We will also use a slight modification of the following compactness
result of Landes, Mustonen \cite{landes-87}, which is an alternative to
the Aubin-Lions lemma in the case of Sobolev spaces.

\begin{proposition}\label{pro:comp}
  Let $p,s \in (1,\infty)$, $q \in [1,p^*)$, where  $p^*:=
  \frac{dp}{d-p}$ if $p<d$, and $p^*:=\infty$ if $p\ge d$. Let  
  $(u_n)$ be a bounded sequence in $L^\infty(0,T;L^1(\Omega))$ such that
  \begin{alignat*}{2}
    u_n&\rightharpoonup u&&\text{ in }L^s(0,T;W^{1,p}_0(\Omega)),\\
   u_n(t)&\rightharpoonup u(t)\quad &&\text{ in }L^1(\Omega)\text{ for almost all }t\in (0,T),
  \end{alignat*}
  than $u_n \to u$ in $L^s(0,T;L^q(\Omega))$.
\end{proposition}
\begin{proof}
  In \cite{landes-87} it is shown that from our assumptions follows
  $u_n \to u$ in $L^s(0,T;L^p(\Omega))$, which is the stated
  assertion if $q\le p$. For $q\in (p,p^*)$  we use this convergence, the interpolation
  $\|v\|_q\le \|v\|_p^{1-\lambda} \|\nabla v\|_{p}^\lambda $, for
  appropriate $\lambda \in (0,1)$ and H\"older's inequality after
  integration in time.
\end{proof}

\begin{proof}[Proof of Theorem \ref{maintheorem}]
  We want to use Proposition \ref{Hirano}. Thus we have to verify all
  conditions in \eqref{Hiranovoraussetzung} for an appropriate
  sequence. To this end we proceed as follows: (i) existence of
  iterates and a~priori estimates, (ii) identification of pointwise
  limit and (iii) verification of condition \eqref{Hiranovoraussetzung}$_4$. 
  
  {\it (i) existence of iterates and a~priori estimates}: For each $M\in \N $
  and each $\tau =\frac TK$, $K \in \N$, we obtain the existence of
  iterates $(u^k_M)_{k=0,\ldots, K} \subseteq V_M$ solving
  \eqref{eq:rg} from Brouwer's fixed point theorem. Using $v_M=u_M^k$
  in \eqref{eq:rg} we obtain in a standard manner the estimate
  \begin{align}
   \label{eq:apriori0}
    \begin{aligned}
      &\frac 12 \|u^\ell_M\|_H^2 +\frac {c_1}{p'} \tau \sum_{k=1}^\ell
      \|u^k_M\|_V^p
      \\
      &\quad \le \frac 12 \|u^0_M\|_H^2 +c_2 \, \tau \sum_{k=1}^\ell
      \|u^k_M\|_{H}^{2} +\frac {c_1^{\frac {-1}{p-1}}}{p'}\, \tau
      \sum_{k=1}^\ell \|f_M(t_k)\|_{V^\ast}^{p'}
    \end{aligned}
  \end{align}
  valid for all $\ell=1,\ldots, K$. Denoting by
  ${\bar{f}}^\tau_M, {\hat{f}}^\tau_M$ the interpolants generated by
  $(f_M(t_k))_{k=0,\ldots, K}$, it follows from Assumption
  \ref{ass:data} that both ${\bar{f}}^\tau_M \to f$ and
  ${\hat{f}}^\tau_M \to f$ in $L^{p'}(0,T;V^\ast)$ as $M\to \infty$,
  $\tau \to 0$. Consequently we get that the first and the last term
  on the right-hand side in \eqref{eq:apriori0} are uniformly bounded
  with respect to $\ell \in \{1,\ldots, K\}$, $M\in \N$ and
  $\tau \le \tau_0$. From discrete Gronwall's inequality we deduce 
  that the left-hand side of \eqref{eq:apriori0} is uniformly bounded
  with respect to $\ell \in \{1,\ldots, K\}$, $M\in \N$ and
  $\tau \le \tau_0$. Thus the interpolants generated by
  $(u^k_M)_{k=0,\ldots, K}$ satisfy for all $M\in \N$,
  $\tau \le \tau_0$
  \begin{align}
    \label{eq:apri1}
    \begin{aligned}
      \|\bar u_M^\tau\|_{L^\infty(0,T;H)}+\|\bar
      u_M^\tau\|_{L^p(0,T;V)} &\le c(T, \|u^0\|_H,
      \|f\|_{L^{p'}(0,T;V^\ast)})\,,
      \\
      \|\hat u_M^\tau\|_{L^\infty(0,T;H)} &\le c(T,\|u^0\|_H,
      \|f\|_{L^{p'}(0,T;V^\ast)})\,.
    \end{aligned}
  \end{align}
  This and Lemma \ref{induzierterOp} imply the existence of sequences
  $M_n \to \infty $, $\tau_n \to 0$ and elements
  $\bar u \in L^\infty(0,T;H)\cap L^p(0,T;V)$,
  $\hat u \in L^\infty(0,T;H)$, $u^\ast \in H$,
  $B^\ast \in L^{p'}(0,T;V^\ast)$ such that $\bar u_n:= \bar
  u_{M_n}^{\tau_n}$, $\hat u_n:= \hat u_{M_n}^{\tau_n}$ satisfy
  \begin{alignat}{2}\raisetag{2cm}\label{konv}
    \begin{aligned}
      \bar u_n &\rightharpoonup \bar u& \quad &\text{in} \
      L^{p}(0,T;V),
      \\
      \bar u_n &\overset{\ast}{\rightharpoonup} \bar u &&\text{in} \
      L^{\infty}(0,T;H),
      \\
      {B}\bar u_n &\rightharpoonup B^\ast &&\text{in} \
      L^{p'}(0,T;V^\ast),
      \\
      \hat u_n &\overset{\ast}{\rightharpoonup} \hat u &&\text{in} \
      L^{\infty}(0,T;H),
      \\
      \bar u_n(T)=\hat u_n (T) &\rightharpoonup u^\ast &&\text{in} \
      H.
    \end{aligned}
  \end{alignat}
  We want to apply Proposition \ref{Hirano} to the sequence $(\bar
  u_n)_{n\in \N}$.

  {\it (ii) identification of pointwise limit}: We have to verify that
  $\bar u_n (t) \rightharpoonup \bar u (t) $ in $H$ for almost all
  $t\in (0,T)$.  Let us first show that $\bar u =\hat u $ in
  $L^2(0,T;H)$. Note that linear combinations of functions of the form
  $\chi_{(s_1,s_2)}(t) v$, where $\chi_{(s_1,s_2)}$, $0<s_1<s_2<T$, is
  the characteristic function of the intervall $(s_1,s_2)$ and
  $v \in H$, are dense in $L^2(0,T;H)$. For $0<s_1<s_2<T$ there exist
  $k_1^n, k_2^n \in \{1,\ldots, K_n\}$,
  $\lambda_1^n, \lambda_2^n \in (0,1]$ such that
  $s_i=\tau_n(\lambda_1^n+k_i^n-1) \in I_{k_i^n}^{\tau_n}$,
  $i=1,2$. Using that
  $\hat u_n(t)-\bar u_n(t)= (u_{M_n}^k-u_{M_n}^{k-1})(\frac t{\tau_n}
  -k)$ on $I_{k}^{\tau_n}$ and \eqref{eq:apri1} we easily see that
  \begin{align*}
    (\hat u_n-\bar u_n, \chi_{(s_1,s_2)}v)_{L^2(0,T;H)}&=
    \int_{s_1}^{s_2} (\hat u_n(t)-\bar u_n(t), v)_H\, dt
    \\
   &\le 4 \tau_n \|\bar u_n\|_{L^\infty(0,T;H)} \|v\|_H \to  0 \quad
     \text{for } {n\to \infty}.
  \end{align*}
  Thus $\hat u_n -\bar u_n \rightharpoonup 0$ in $L^2(0,T;H)$, which
  implies $\bar u =\hat u $ in $L^2(0,T;H)$, and thus also in $L^\infty(0,T;H)$.

  Next, notice that \eqref{eq:rg} can for all $v \in V_{M_n}$ and almost
  all $t \in (0,T)$ be re-written as
  \begin{align}
    \label{eq:ga}
    \begin{aligned}
      \Big \langle\frac {d \hat u_{n}  (t)}{dt},v\Big \rangle_V +
      \langle B \bar u_{n} (t) ,v\rangle _V= \langle
      \bar f_{n}(t),v\rangle_V ,
     \end{aligned}
  \end{align}
  where $\bar f_n$ is the piecewise constant interpolant generated by
  $(f_{M_n}(t_k^{\tau_n}))_{k=0,\ldots, K_n}$.  For an arbitrary $s\in (0,T)$
  let $\phi_s\in C^\infty_0(0,T)$ satisfy $0\le \phi_s\le 1$ and  $\phi_s\equiv 1 $ in a
  neighborhood of $s$. Let $k\in\N$ and let $m,n\in \N $ be such that
  $M_n, M_m \ge k$. Multiplying \eqref{eq:ga} for an arbitrary
  $v\in V_k$ by $\phi_s$, integrating over $(0,s)$ with respect to
  $t$, using the integration by parts formula and the properties of
  the Gelfand triple we obtain
  \begin{align}
    &(\hat u_n(s)-\hat u_m(s),v)_H \label{eq:cauchy}\\
    &=\int_0^s(\hat u_n(t)-\hat u_m(t),v)_H\,\phi'_s(t)\,dt
      -\int_0^s\langle B\bar u_n(t)-B\bar
      u_m(t),v\rangle_V\,\phi_s(t)\,dt \notag\\
    &\quad +\int_0^s\langle
      \bar f_n(t)-\bar f_m(t),v\rangle_V\,\phi_s(t)\,dt\,.\notag 
  \end{align}
  In view of \eqref{konv} and $\bar f_n \to f$ in $L^{p'}(0,T;V^\ast)$
  we see that the right-hand side converges to $0$ for $n,m \to
  \infty$. Since $\bigcup _{k\in \N}V_k$ is dense in $H$, this shows
  that for every $s\in (0,T)$ the sequence $(\hat u_n(s))_{n\in \N}$ is a weak
  Cauchy sequence in $H$. Thus, for every $s \in (0,T)$ there exists
  $w(s) \in H$ such that $\hat u_n(s) \rightharpoonup 
  w(s)$ in $H$. From this, \eqref{eq:apri1} and the Lebesgue theorem
  on dominated convergence follows for all $\phi \in L^2(0,T;H)$
  \begin{align*}
    \lim_{n\to \infty} \int_0^T (\hat u_n(t), \phi (t))_H\, dt
    =\int_0^T ( w(t), \phi (t))_H\, dt .
  \end{align*}
  This together with \eqref{konv}$_4$ implies $w =\hat u $ in
  $L^2(0,T;H)$. Since $\bar u =\hat u $ in $L^2(0,T;H)$ we proved
  for almost every $t \in (0,T)$
  \begin{align}\label{eq:point}
    \hat u_n(t) \rightharpoonup \bar  u(t) \qquad \text{in } H\,.
  \end{align}
  However we need to prove $\bar u_n(t) \rightharpoonup \bar u(t)$ in
  $H$ for almost all $t \in (0,T)$. To this end we proceed as follows:
  For given $m\in \N$ let $n\ge m$ be arbitrary. Then we have, using
  that  $\hat u_n(t)-\bar u_n(t)= d_\tau u_{M_n}^k(t -k\tau_n)$ on $I_{k}^{\tau_n}$
  \begin{align*}
    \|\hat u_n -\bar u_n \|_{L^{p'}(0,T; V_m^\ast)}^{p'} 
    &\le     \|\hat u_n -\bar u_n \|_{L^{p'}(0,T; V_n^\ast)}^{p'}
    \\
    &= \sum_{k=1}^{K_n} \|d_\tau  u^k_{M_n}\|_{V^\ast_n}^{p'}\int_{I^{\tau
      _n}_k} |t -k\tau_n|^{p'}\, dt
    \\
    &= \frac {\tau_n^{p'}}{p'+1}\tau_n \sum_{k=1}^{K_n} \|d_\tau  u^k_{M_n}\|_{V_n^\ast}^{p'}.
  \end{align*}
  The equations \eqref{eq:rg} yield 
  \begin{align*}
    \|d_\tau  u^k_{M_n}\|_{V_n^\ast} \le \|f_{M_n}(t_k^{\tau_n})\|_{V^\ast} +\|B u_{M_n}^k \|_{V^\ast}\,,
  \end{align*}
  and thus
  \begin{align*}
    \|\hat u_n -\bar u_n \|_{L^{p'}(0,T; V_m^\ast)}^{p'}  \le  \frac
    {\tau_n^{p'}}{p'+1} \big ( \|\bar
    f_{n}\|_{L^{p'}(0,T;V^\ast)}^{p'} +\|B \bar u_{n}
    \|_{L^{p'}(0,T;V^\ast)}^{p'}\big) \,,
  \end{align*}
  which converges to $0$ in view of \eqref{konv} and and $\bar f_n \to
  f$ in $L^{p'}(0,T;V^\ast)$. Applying a diagonal procedure we get 
  for all $m\in \N$ and almost all $t \in (0,T)$ that 
  \begin{align*}
    \hat u_n(t) -\bar  u_n(t) \to 0 \qquad \text{in } V_m^\ast\,,
  \end{align*}
  which together with \eqref{eq:point}, the properties of the Gelfand
  triple and the density of $\bigcup _{k\in \N}V_k$ in $H$ yields 
  \begin{align}
    \label{eq:point2}
    \bar u_n(t) \rightharpoonup \bar  u(t) \qquad \text{in } H\,.
  \end{align}

  {\it (iii) verification of condition
    \eqref{Hiranovoraussetzung}$_4$}:
  From \eqref{eq:ga} and the integration by parts formula we obtain for
  all $\phi \in C_0^\infty (\R)$ and all $v\in V_m$, where $M_n \ge m$
  \begin{align*}
    &(\hat u_n(T),v)_H\phi(T)-(\hat u_n(0),v)_H\phi(0)\\
    &=\int_0^T(\hat u_n(t),v)_H\,\phi'(t) 
      -\langle B\bar u_n(t),v\rangle_V\,\phi(t) + \langle
      \bar f_n(t),v\rangle_V\,\phi(t)\,dt\,.
  \end{align*}
  In view of \eqref{konv}, $\bar f_n \to f$ in $L^{p'}(0,T;V^\ast)$
  the density of $\bigcup _{k\in \N}V_k$  in $V$ and $\bar u =\hat u$
  in $L^2(0,T;H)$ we obtain 
  \begin{align*}
    &( u^\ast,v)_H\phi(T)-(u^0,v)_H\phi(0)\\
    &=\int_0^T(\bar u(t),v)_H\,\phi'(t) 
      -\langle B^\ast(t),v\rangle_V\,\phi(t) + \langle
      \bar f(t),v\rangle_V\,\phi(t)\,dt
  \end{align*}
  for all $\phi \in C_0^\infty (\R)$ and all $v\in V$. For $\phi
  \in C_0^\infty (0,T)$ this and the definition of the generalized
  time derivative  imply
  \begin{align}
    \label{eq:ut}
    \frac{d\bar u }{dt} = f - B^\ast \qquad \text{in } L^{p'}(0,T;
    V^\ast). 
  \end{align}
  Moreover, by standard arguments we get $\bar u \in C(\bar I; H)$,
  $u^\ast =\bar u(T)$, and  $\hat u_n(T) =\bar u_n(T) \rightharpoonup
  \bar u(T)$ in $H$. Using \eqref{eq:ga} for $v=\bar u_n(t)$ and 
  \begin{align*}
      \Big \langle \frac {d \hat u_{n}  }{dt},\bar u_n\Big \rangle
    _{L^p(0,T;V)} = \tau _n \sum _{k=1}^{K_n} (d_\tau  u_{M_n}^k,
    u_{M_n}^k)_H \ge \frac 12 \|\bar u _n(T)\|_H^2 -
    \frac 12 \|u _n^0\|_H^2 
  \end{align*}
  we obtain 
  \begin{align*}
      \langle B \bar u_{n}  ,\bar u_n\rangle _{L^p(0,T;V)} &= \langle
      \bar f_{n}, \bar u_n\rangle_{L^p(0,T;V)} - \Big \langle \frac {d \hat u_{n}  }{dt},\bar u_n\Big \rangle
    _{L^p(0,T;V)}  
    \\
    &\le \langle \bar f_{n}, \bar u_n\rangle_{L^p(0,T;V)} + \frac 12
      \|u _n^0\|_H^2 - \frac 12 \|\bar u _n(T)\|_H^2 .
  \end{align*}
  Thus \eqref{konv}, $\bar f_n \to f$ in $L^{p'}(0,T;V^\ast)$ and the
  lower weak semicontinuity of the norm imply
  \begin{align*}
      \limsup _{n\to \infty} \langle B \bar u_{n}  ,\bar u_n\rangle _{L^p(0,T;V)} 
    &\le \langle f, \bar u\rangle_{L^p(0,T;V)} + \frac 12
      \|u ^0\|_H^2 - \frac 12 \|\bar u (T)\|_H^2.
  \end{align*}
  From \eqref{eq:ut}, the integration by parts formula and
  \eqref{konv} we get
  \begin{align*}
    \langle f, \bar u\rangle_{L^p(0,T;V)} = \frac 12 \|\bar u
    (T)\|_H^2 -\frac 12 \|u ^0\|_H^2 + \lim_{n \to \infty}
    \langle B \bar u_{n}  ,\bar u\rangle _{L^p(0,T;V)} \,. 
  \end{align*}
  The last two inequalities imply that also condition
  \eqref{Hiranovoraussetzung}$_4$ is satisfied.

  Thus, we have verified all conditions in \eqref{Hiranovoraussetzung}
  and consequently Proposition~\ref{Hirano} together with \eqref{konv} implies $B^\ast = B\bar u$
  in $L^{p'}(0,T;V^\ast)$. This and \eqref{eq:ut} yield
  \begin{align*}
    \frac{d\bar u }{dt}  + B\bar u = f  \qquad \text{in } L^{p'}(0,T;
    V^\ast),
  \end{align*}
  i.e.~$\bar u $ is a solution of \eqref{eq:pm}.  
\end{proof}

\section{Convergence of a semi-implicit scheme}

For a given N-function $\varphi$ having $(p,\delta)$-structure we address
the following evolution problem
\begin{align}\label{eq:mp}
  \begin{aligned}
    \frac{du}{dt}(t) -\divo A_0(\nabla u(t))+ g(u(t))&=f &&\text{in }V^\ast \text{ for a.e.} \ t \in [0,T]\\
    u(0)&=u^0 &&\text{in}\ H,
  \end{aligned}
\end{align}
where $A_0$ is given by \eqref{eq:op-a} for $\alpha=0$ and
$g\colon \R \to \R$ is a given function. Concerning the function $g$
we make the following assumption:
\begin{ass}[Nonlinearity]\label{ass:non}
  Let the function $g \colon \R \to \R$ be given by $g(s):=\h(s)\,s$, $s\in \R$, 
  with a continuous function  $\h\colon \R\to \R$ that satisfies:
  \begin{itemize}
  \item[\textnormal{(H1)}] There exists a constant $c_7 >0$ such that
    for all $s\in \R$
    \begin{align*}
      \h(s) \ge -c_7\,.
    \end{align*}
  \item[\textnormal{(H2)}] There exists $r\in (2,\infty)$ and a
    constant $c_8>0$ such that for all $s \in \R$
    \begin{align*}
      \abs{\h(s)} \leq c_8 \big ( 1 + \abs{s} ^{r-2}\big )\,.
    \end{align*}
\end{itemize}
\end{ass}
Note that (H2) implies that there exists a constant $c_9=c_9(r,c_8)>0$ such that
for all $s \in \R$
\begin{align}\label{eq:growth}
  \abs{g(s)} \leq c_9 \big ( 1 + \abs{s} ^{r-1}\big )\,.
\end{align} 
In what follows we abbreviate 
$$
  V:=W^{1,p}_0(\Omega)\quad \text{ and } \quad H:=L^2(\Omega).
$$
The N-function $\vphi$ and the functions $g, \h $ induce operators
$A\colon V\to V^*$,
$G\colon L^q(\Omega) \to L^{\frac q{r-1}}(\Omega)$, $q\in [1,\infty)$,
and $\gH \colon L^q(\Omega) \to L^{\frac q{r-2}}(\Omega)$,
$q\in [\max\{1,r-2\}, \infty)$ via 
\begin{align}
  \label{eq:ops}
  \begin{aligned}
    \langle A u,v\rangle_V &:= \int_\Omega A_0(\nabla u) \cdot \nabla
    v\, \dv{x}\,,
    \\
    (Gu)(x)&:= g(u(x))\,,
    \\
    (\gH u)(x)&:= \h(u(x))\,.
  \end{aligned}
\end{align}
\begin{lemma}\label{lem:op}
  Let $\vphi$ have $(p,\delta)$-structure for some $p \in (1,\infty)$
  and ${\delta\ge 0}$ and let the Assumption~\ref{ass:non} be
  satisfied. Then the operators $A\colon V\to V^*$,
  $\gH \colon L^q(\Omega) \to L^{\frac q{r-2}}(\Omega)$,
  $q\in [\max\{1,r-2\}, \infty)$, and
  $G\colon L^q(\Omega) \to L^{\frac q{r-1}}(\Omega)$,
  $q\in [1,\infty)$ defined in \eqref{eq:ops} are continuous and
  bounded. Moreover, the operator $A$ is strictly monotone and
  coercive. In particular, the operator $B\colon V\to V^*$ defined
  via $Bu:=Au +Gu$ satisfies Assumption~\ref{voroperatorfamilie}
  if $p>\frac{2d}{d+2}$ and $r \in (2,p\frac {d+2}d]$.
\end{lemma}
\begin{proof}
  Since $V\eqsim W^{1,\vphi}_0(\Omega)$ the properties of $A$ follow
  from the properties of $\vphi$ in a standard manner. Thus the
  operator $A$ satisfies Assumption~\ref{voroperatorfamilie} with
  constants $c_1=c_1(p)$, $c_3=c_3(p)\delta^p$,
  $c_6=c_6(p,\abs{\Omega})\delta^{p-1}$ and $c_4=c_4(p)$,
  $c_2=c_5=0$. From Assumption~\ref{ass:non} we deduce that $H, G$ are
  Nemyckii operators, for which the stated properties follow in a
  standard way. Moreover, for ${r\in (2,p^*)}$, recall that 
 $p^*= \frac{dp}{d-p}$ if $p<d$, and $p^*=\infty$ if $p\ge d$, the operator
  $G\colon V \to V^*$ is compact, since the embedding
  $V\hookrightarrow\hookrightarrow L^r(\Omega)$ is compact. Thus we get that the operator $B$ is
  pseudomonotone. For $p>\frac{2d}{d+2}$ we get that $(V,H,V^*)$ forms
  a Gelfand triple and that $p\frac{d+2}d <p^*$. The
  Assumption~\ref{ass:non}, H\"older's inequality,
  interpolation, embeddings  and $r\le p\frac {d+2}d$ (cf.~\cite{br-hirano} for more
  details) imply that $G$ satisfies (A2), (A3) with
  constants $c_2=c_7$, $c_4=c_9$,   $c_1=c_3=c_4=c_6=0$. Consequently,
  $B$ satisfies Assumption~\ref{voroperatorfamilie}. 
\end{proof}


In view of this lemma we can apply Theorem~\ref{maintheorem} to the
present situation if we make analogous assumptions on the data to
Assumption~\ref{ass:data}. The assumption applies to standard finite
element methods on polyhedral Lipschitz domains (cf.~\cite{BreSco08}). 

\begin{ass}[Data I]\label{ass:data2}
  Let $p>\frac {2d}{d+2}$ and let $u^0\in H$ and $f \in L^{p'}(0,T;V^\ast)$ be given. 
  Let $V_h \subset W^{1,\infty}_0(\Omega)$, $h>0$, be conforming finite
  element spaces, corresponding to shape regular triangulations
  $\mathcal T_h$. We equip $V_h$ with the \mbox{$V$-norm} and assume that
  $V_{ h/2 }\subset V_h$ and that $\bigcup _{m\in \N}V_{h2^{-m}}$
  is dense in $V$. We assume that there exists a sequence
  $(u^0_h) \subset V_h$ with $u^0_h \to u^0$ in $H$. For each
  $\eps >0$ we set $\ue ^0_h:=u^0_h$. We further assume that there exists a
  sequence $(f_h)\subset C(0,T;V^*)$ such that $f_h\to f$ in
  $L^{p'}(0,T;V^*)$.
\end{ass}

Let us first  study an implicit scheme.
Let $\eps \in [0,1)$. For given $h>0$ and $\ue^0_h \in V_h$ the
sequence of iterates $(\prescript{\eps}{}u^k_h )_{k=0,\ldots , K}
\subseteq V_h$ is given via 
\begin{align}\label{eq:phi-rgi}
  \begin{aligned}
    &\big( d_\tau \ue^k_h,v_h\big ) +
    \Big(\frac{\vphi_\eps'(|\nabla \ue^{k}_h|)}{|\nabla
      \ue^{k}_h|}\nabla \ue^k_h,\nabla v_h \Big)
    \\
    &\quad + (\h(\ue^{k}_h)\, \ue^k_h,v_h)= (f_h(t_k),v_h) \qquad
    \forall v_h \in V_h.
  \end{aligned}
\end{align}
\begin{theorem}[Convergence of the implicit scheme]\label{maintheorem0}
  Let $\vphi$ have $(p,\delta)$-structure for some
  $p \in (\frac{2d}{d+2},\infty)$ and ${\delta\ge 0}$, let
  Assumption~\ref{ass:non} be satisfied for some
  $r \in (2,p\frac {d+2}d]$ and let Assumption~\ref{ass:data2} be
  satisfied. Let
  $\bar u_n:=\prescript{\eps _n}{}{\bar u}_{h_n}^{\tau_n}$ be a
  sequence of piecewise constant interpolants generated by iterates
  $(\prescript{\eps _n}{}u^k_{h_n})_{k=0,\ldots, K_n}$,
  $K_n=\frac T{\tau_n}$, solving \eqref{eq:phi-rgi} for some sequences
  $h_n\to 0$, $\tau_n\to 0$, $\eps _n \to 0$. Then each weak$^\ast $
  accumulation point $u$ of the sequence $(\bar u_n)_{n\in \N}$ in the
  space $L^\infty(0,T;H)\cap L^p(0,T;V)$ belongs to the space
  $ W_p^1(0,T;V,H)$ and is a solution of \eqref{eq:mp}.
\end{theorem}
\begin{proof}
  In the case $\eps =0 $ we choose $\eps_n =0$ and the statement of
  the theorem follows from Theorem~\ref{maintheorem}. In the case
  $\eps>0$ we have to re-write the scheme \eqref{eq:phi-rgi} as
  \begin{align}\label{eq:phi-rgi-new}
    \begin{aligned}
      &\langle d_\tau \ue^k_h,v_h\rangle _V + \big( A_0(\nabla
      \ue_h^k), \nabla v_h \big ) + (\h(\ue^{k}_h)\, \ue^k_h,v_h)
      \\
      &= (f_h(t_k),v_h) + \big (\Ee_h^k, \nabla v_h \big),
    \end{aligned}
  \end{align}
  where
  \begin {align*}
     \big (\Ee_h^k, \nabla v_h  \big)&:= \big(A_0(\nabla \ue_h^k) -
     A_\eps (\nabla \ue_h^k), \nabla v_h \big )\,.
  \end{align*}
  The proof of the assertion now follows along the lines of the proof
  of Theorem~\ref{maintheorem}. The additional term $\Ee_h^k$ can be
  treated due to Lemma~\ref{lem:eps}. We omit the details here, since
  they will be discussed in detail in the proof of
  Theorem~\ref{maintheorem1}, where the same term occurs.  
\end{proof}

In the scheme \eqref{eq:phi-rgi} we still have to solve nonlinear
equations. If we want to avoid this and only solve linear equations we
can study the following semi-implicit scheme:
Let $\eps \in (0,1)$. For given $h>0$ and $\ue^0_h \in V_h$ the
sequence of iterates $(\prescript{\eps}{}u^k_h )_{k=0,\ldots , K}
\subseteq V_h$ is given via 
\begin{align}\label{eq:phi-rg}
  \begin{aligned}
    &\big ( d_\tau \ue^k_h,v_h\big )+ 
    \Big(\frac{\vphi_\eps'(|\nabla \ue^{k-1}_h|)}{|\nabla
      \ue^{k-1}_h|}\nabla \ue^k_h,\nabla v_h \Big)
    \\
    &\quad + (\h(\ue^{k-1}_h)\, \ue^k_h,v_h)= (f_h(t_k),v_h) \qquad
    \forall v_h \in V_h.
  \end{aligned}
\end{align}
To show that also this scheme converges to  a weak
solution of \eqref{eq:mp} we have to make more restrictive assumptions on the data.

\begin{ass}[Data II]\label{ass:data1}
  Let $p>\frac {2d}{d+2}$ and let $u^0\in V$ and $f \in L^{p'}(0,T;H)$
  be given.  Let $V_h \subset W^{1,\infty}_0(\Omega)$, $h>0$, be
  conforming finite element spaces, corresponding to shape regular
  triangulations $\mathcal T_h$. We equip $V_h$ with the
  \mbox{$V$-norm} and assume that $V_{ h/2 }\subset V_h$ and that
  $\bigcup _{m\in \N}V_{h2^{-m}}$ is dense in $V$. We assume that
  there exists a sequence $(u^0_h) \subset V_h$ with $u^0_h \to u^0$
  in $V$. For each $\eps >0$ we set $\ue ^0_h:=u^0_h$. We assume that
  there exists a sequence $(f_h)\subset C(0,T;H)$ such that $f_h\to f$
  in $L^{p'}(0,T;H)$.
\end{ass}

The following theorem excludes the special case $p=2$ which is discussed
in a subsequent remark. 

\begin{theorem}[Convergence of the semi-implicit scheme]\label{maintheorem1}
  Let $\varphi$ have $(p,\delta)$-structure for some $p \in
  (\frac{2d}{d+2},2)$ 
  and   $\delta\ge 0$, let 
  Assumption~\ref{ass:non} be satisfied for some
  $r \in (2,p\frac {d+2}{2d}+1]$ and let Assumptions~\ref{ass:data2} be
  satisfied. Let
  $\bar u_n:=\prescript{\eps _n}{}{\bar u}_{h_n}^{\tau_n}$ be a
  sequence of piecewise constant interpolants generated by iterates
  $(\prescript{\eps _n}{}u^k_{h_n})_{k=0,\ldots, K_n}$,
  $K_n=\frac T{\tau_n}$, solving \eqref{eq:phi-rg} for some sequences
  $h_n\to 0$, $\tau_n\to 0$, $\eps _n \to 0$ satisfying
  $\tau_n=o(\varphi''(\eps_n))$.  Then each weak$^\ast $ accumulation point $u$ of the sequence
  $(\bar u_n)_{n\in \N}$ in the space $L^\infty(0,T;V)$
  belongs to the space $ W_p^1(0,T;V,H) \cap L^\infty (0,T;V)$ and is a solution of
  \eqref{eq:mp}.
\end{theorem}
\begin{proof}
  In order to adapt the arguments of the proof of
  Theorem~\ref{maintheorem} to the present situation we re-write
  \eqref{eq:phi-rg} as an implicit scheme with resulting error
  terms on the right-hand side. The handling of these new terms in the
  verification of the conditions in \eqref{Hiranovoraussetzung} is
  possible due to a second a~priori estimate, obtained by testing with
  the backward difference quotient of the solution. For the
  verification of the last condition in \eqref{Hiranovoraussetzung} we
  also use the compactness argument in Proposition~\ref{pro:comp}.

  {\it (i) existence of iterates and a~priori estimates}: For each
  $h>0$, $\eps \in (0,\eps_0)$, where we assume without loss of generality that $\veps_0=1$,   
   and each $\tau =\frac TK$, $K \in \N$, the existence of iterates
  $(\ue^k_h)_{k=0,\ldots, K} \subseteq V_h$ solving \eqref{eq:phi-rg}
  is clear since these are linear equations. Using $v_h=\ue_h^k$ in
  \eqref{eq:phi-rg} we obtain, also using the Assumption~\ref{ass:non}
  and Young's inequality, the estimate
  \begin{align}
    \label{eq:apriori}
    \begin{aligned}
      &\frac 12 \|\ue^\ell_h\|_H^2 + \tau \sum_{k=1}^\ell \int_\Omega
      \frac{\vphi_\eps'(|\nabla \ue^{k-1}_h|)}{|\nabla
        \ue^{k-1}_h|}|\nabla \ue^k_h|^2 \, \dv{x}
      \\
      &\quad \le \frac 12 \|u^0_h\|_H^2  +(c_7+1) \, \tau \sum_{k=1}^\ell
      \|u^k_M\|_{H}^{2} + \tau
      \sum_{k=1}^\ell \|f_h(t_k)\|_{H}^{2}
    \end{aligned}
  \end{align}
  valid for all $\ell=1,\ldots, K$. Due to Assumption \ref{ass:data1} the
  first and last term on the right-hand side of \eqref{eq:apriori} are
  uniformly bounded with respect to $h >0$, $\tau, \eps \in(0, 1)$ and
  $\ell\in\{1,\ldots,K\}$. Thus discrete Gronwall's inequality yields
  that the left-hand side of \eqref{eq:apriori} is uniformly bounded
  with respect to $h>0$, $\tau, \eps \in(0, 1)$ and
  $\ell\in\{1,\ldots,K\}$. In particular we get that 
  interpolants generated by $(\ue^k_h)_{k=0,\ldots, K}$ satisfy for
  all $h>0$, $\tau, \eps \in (0, 1)$
  \begin{align}
    \label{eq:apri10}
    \begin{aligned}
      \|\bue_h^\tau\|_{L^\infty(0,T;H)}  + \|\hue_h^\tau\|_{L^\infty(0,T;H)}
      &\le c(\|u^0\|_H , \|f\|_{L^2(0,T;H)})\,. 
    \end{aligned}
  \end{align}
  Using $v_h =d_\tau \ue_h^k $ and Lemma~\ref{la:orlicz_stab} we
  obtain in the same way as in~\cite{BDN}, using also Young's
  inequality,
  \begin{align}
    \label{eq:ener_bound1}
    \begin{aligned}
      &E_{\vphi_\eps}[\ue^\ell_h] + \frac \tau 2 \sum_{k=1}^\ell \|d_\tau
      \ue^k_h\|^2_H + \frac{\tau^2}{2} \sum_{k=1}^\ell \int_\O
      \frac{\vphi_\eps'(|\nabla \ue^{k-1}_h|)}{|\nabla \ue^{k-1}_h|}
      |d_\tau \nabla \ue^k_h|^2 \dv{x}
      \\
      &\quad \le E_{\vphi_\eps}[u_h^0] +  \tau  \sum_{k=1}^\ell
      \|f_h(t_k)\|^2_H + \tau \sum_{k=1}^\ell \int_\Omega
      \abs{\h(\ue^{k-1}_h)}^2\, \abs{\ue^k_h}^2 \,\dv{x}\,,
    \end{aligned}
  \end{align}
  valid for all $\ell=1,\ldots, K$. Due to Assumption \ref{ass:data1} the
  first two terms on the right-hand side are uniformly bounded with
  respect to $h >0$ and $\eps,\tau  \in (0,1)$. Moreover, using
  \eqref{eq:equi} we get
  \begin{align}
    \label{eq:lower}
    E_{\vphi_\eps }[v] \ge c\,\big ( \norm{v}_V^p - \eps^p -\delta^p \big )\,.
  \end{align}
  The Assumption (H2), Young's inequality, the interpolation of
  $L^{2(r-1)}(\Omega)$ between $H$ and $V$ and \eqref{eq:apri10}
  yield
  \begin{align}\label{eq:gron}
    \begin{aligned}
      \int_\Omega\abs{\h(\ue^{k-1}_h)}^2\, \abs{\ue^k_h}^2 \,\dv{x}
      &\le c\Big ( \|\ue^k_h\|_H^2 + \|\ue^k_h\|_{2(r-1)}^{2(r-1)}+
      \|\ue^{k-1}_h\|_{2(r-1)}^{2(r-1)}\Big )
      \\
      &\le c \big ( 1+ \|\ue^k_h\|_{V}^{p \frac{2d(r-2)}{p(d+2)-2d}}+
      \|\ue^{k-1}_h\|_{V}^{p \frac{2d(r-2)}{p(d+2)-2d}}\big)\,.
    \end{aligned}
  \end{align}
  Requiring that ${p \frac{2d(r-2)}{p(d+2)-2d}}\le 1$ we get the
  restriction $r\le p \frac{d+2}{2d} +1$. The last estimate together
  with \eqref{eq:lower}, \eqref{eq:apri10}, \eqref{eq:ener_bound1} and  discrete
  Gronwall's inequality yield  that the interpolants generated by
  $(\ue^k_h)_{k=0,\ldots, K}$ and the piecewise constant interpolant generated by
  $(\prescript{\eps}{} u_{h}^{k-1})_{k=0,\ldots, K}$, which we denote
  by $\prescript{\eps}{} {\tilde u}_{h}^{\tau}$, satisfy for
  all $h>0$, $\tau, \eps \in (0, 1)$
  \begin{align}
    \label{eq:apri2}
    \begin{aligned}
      \|\prescript{\eps}{} {\tilde u}_{h}^{\tau}\|_{L^\infty(0,T;V)}+
      \|\bue_h^\tau\|_{L^\infty(0,T;V)}
      &\le c(\delta,p,T, \abs{\Omega},\|f\|_{L^2(0,T;H)} ,\|u^0\|_V)\,, 
      \\
      \Big\|\frac{d \hue_h^\tau}{dt}\Big \|_{L^2(0,T;H)}+
      \|\hue_h^\tau\|_{L^\infty(0,T;V)}
      &\le c(\delta,p,T, \abs{\Omega},\|f\|_{L^2(0,T;H)} ,\|u^0\|_V)\,.
    \end{aligned}
  \end{align}

  Using Assumptions (A3) and (H2) one can show
  (cf.~Lemma~\ref{induzierterOp}) that the induced operators
  $A,B,\gH,G$ are bounded operators in the following settings:
  $A\colon L^\infty(0,T;V) \to L^\infty(0,T;V^*)$,
  $B\colon L^\infty(0,T;V) \to L^\infty(0,T;V^*)$,
  $G\colon L^\infty(0,T;V) \to L^\infty(0,T;
  L^{\frac{p^*}{r-1}}(\Omega))$,
  $\gH\colon L^\infty(0,T;V) \to L^\infty(0,T;
  L^{\frac{p^*}{r-2}}(\Omega) )$.  For later purposes we now choose
  $\tau=o(\vphi''(\eps))$.  Thus \eqref{eq:apri2} and the last
  observation imply the existence of sequences ${h_n \to 0 }$,
  $\tau_n \to 0$, $\eps _n \to 0$ and elements $u^\ast \in H$,
  $\bar u \in L^\infty(0,T;V)$, $\hat u \in L^\infty(0,T;V)$,
  ${\tilde u \in L^\infty(0,T;V)}$,
  $A^\ast \in L^\infty(0,T;V^\ast)$,
  ${\gH^\ast \in L^\infty(0,T; L^{\frac{p^*}{r-1}}(\Omega))}$ such
  that $\bar u_n:= \prescript{\eps _n\hspace{-.25mm}}{}{\bar u}_{h_n}^{\tau_n}$,
  $\hat u_n:= \prescript{\eps _n\hspace{-.25mm}}{}{\hat u}_{h_n}^{\tau_n}$,
  $\tilde u_n:= \prescript{\eps_n\hspace{-.25mm}}{} {\tilde u}_{h_n}^{\tau_n}$
  satisfy 
  \begin{alignat}{2}\raisetag{2cm}\label{konv1}
    \begin{aligned}
      \bar u_n &\overset{\ast}{\rightharpoonup} \bar u &&\text{in} \
      L^{\infty}(0,T;V),
      \\
      \hat u_n &\overset{\ast}{\rightharpoonup} \hat u &&\text{in} \
      L^{\infty}(0,T;V),
      \\
      \tilde u_n &\overset{\ast}{\rightharpoonup} \tilde u &&\text{in} \
      L^{\infty}(0,T;V),
      \\
      {A}\bar u_n &\overset{\ast}{\rightharpoonup} A^\ast &&\text{in} \
     L^{\infty}(0,T;V^\ast),
      \\
      {\gH}(\tilde u_n)\bar u_n &\overset{\ast}{\rightharpoonup} \gH^\ast &&\text{in} \
      L^{\infty}(0,T; L^{\frac{p^*}{r-1}}(\Omega)) \cap L^{\infty}(0,T;V^\ast),
      \\
      \bar u_n(T)=\hat u_n (T) &\rightharpoonup u^\ast &&\text{in} \
      H.
    \end{aligned}
  \end{alignat}
  We want to apply Proposition \ref{Hirano} to the sequence $(\bar
  u_n)_{n\in \N}$ and the operator $B\colon V\to V^*$ defined via $Bv
  := Av+\gH (v)v$ (cf.~Lemma~\ref{lem:op}).

  {\it (ii) perturbed implicite scheme}: To adapt the arguments from the proof of Theorem~\ref{maintheorem}
  to the present situation, we re-write the scheme \eqref{eq:phi-rg}
  for all $v_h\in V_h$ as a perturbed implicite scheme 
  \begin{align}\label{eq:phi-rg-new}
    \begin{aligned}
      &\langle d_\tau \ue^k_h,v_h\rangle _V + \big( A_0(\nabla
      \ue_h^k), \nabla v_h \big ) + (\h(\ue^{k-1}_h)\, \ue^k_h,v_h)
      \\
      &= (f_h(t_k),v_h) + \big (\Ee_h^k, \nabla v_h \big) + \big
      (\Fe_h^k, \nabla v_h \big),
    \end{aligned}
  \end{align}
  where
  \begin {align*}
     \big (\Ee_h^k, \nabla v_h  \big)&:= \big(A_0(\nabla \ue_h^k) -
     A_\eps (\nabla \ue_h^k), \nabla v_h \big )\,,
    \\
     \big (\Fe_h^k, \nabla v_h  \big) &:= \Big (A_\eps (\nabla \ue_h^k)-
    \frac{\vphi_\eps'(|\nabla \ue^{k-1}_h|)}{|\nabla \ue^{k-1}_h|}\nabla
    \ue^k_h,\nabla v_h \Big)\,.
  \end{align*}
  To verify the conditions \eqref{Hiranovoraussetzung} we proceed as
  in the proof of Theorem~\ref{maintheorem}. In the following we
  concentrate on the treatment of the new terms.

  {\it (iii) identification of the pointwise limit}: In view of \eqref{eq:apri2} we can prove in the same way as in the
  proof of Theorem \ref{maintheorem} that $\bar u =\hat u $ in
  $L^2(0,T;H)$, and thus also in $L^\infty(0,T;V)$. From
  \eqref{eq:apri2} follows 
  \begin{align}\label{eq:u=u}
   \int_0^T    \norm {\tilde u_n -\bar u_n}_H^2\, \dv{t} = \tau_n^2\,
    \Big\|\frac{d \prescript{\eps_n}{} {\hat u}_{h_n}^{\tau_n}}{dt}\Big \|^2_{L^2(0,T;H)} \to 0\,,
  \end{align}
  which implies that also $\tilde u = \bar u$ in $L^\infty(0,T;V)$.

  Next, notice that \eqref{eq:phi-rg-new} can for all $v \in V_{h_n}$ and almost
  all $t \in (0,T)$ be re-written as
  \begin{align}
    \label{eq:ga1}
    \begin{aligned}
      &\Big \langle\frac {d \hat u_{n}  (t)}{dt},v\Big \rangle_V +
      \langle A \bar u_{n} (t) ,v\rangle _V + (\gH(\tilde u_n)(t)\bar u_n(t),
      v)_H
      \\
      &= (\bar f_n(t),v)_H+\langle E_n(t),v\rangle
      _V +\langle F_n(t),v\rangle _V , 
     \end{aligned}
  \end{align}
  where 
  \begin{align*}
    \langle E_n(t),v\rangle _V
    &:= \big(A_0(\nabla \bar u_n(t)) -  A_{\eps_n} (\nabla \bar u_n(t)), \nabla v \big )_H
    \\
    \langle F_n(t),v\rangle _V
    &:= \Big ( \frac{\vphi_{\eps _n}'(\abs{\nabla
      \bar u_n(t)})}{\abs{\nabla \bar u_n(t)}} \nabla \bar u_n(t) -\frac{\vphi_{\eps _n}'(\abs{\nabla
      \tilde u_n(t)})}{\abs{\nabla \tilde u_n(t)}} \nabla \bar u_n(t),
      \nabla v \Big )_H\,,
  \end{align*}
  where $\bar f_n$ is the piecewise constant interpolant generated by
  $(f_{h_n}(t_k^{\tau_n}))_{k=0,\ldots, K_n}$. Similarly to the
  derivation of \eqref{eq:cauchy} we obtain for an arbitrary
  $s\in (0,T)$, an arbitrary $k \in \N$, $m,n \ge k$, and all
  $v \in V_{h_k}$, all $\phi_s\in C^\infty_0(0,T)$ satisfying
  $\phi_s\equiv 1 $ in a neighborhood of $s$
  \begin{align*}
    &(\hat u_n(s)-\hat u_m(s),v)_H
    \\
    &=\int_0^s(\hat u_n(t)-\hat u_m(t),v)_H\,\phi'_s(t) 
      -\langle A\bar u_n(t)-A\bar u_m(t),v\rangle_V\,\phi_s(t)\,dt
    \\
    &\quad +\int_0^s \big(A_0(\nabla \bar u_n(t)) -  A_{\eps_n}
      (\nabla \bar u_n(t)), \nabla v \big )_H \, \phi_s(t)\,dt
    \\
    &\quad -\int_0^s \big(A_0(\nabla \bar u_m(t)) -  A_{\eps_m}
      (\nabla \bar u_m(t)), \nabla v \big )_H \, \phi_s(t)\,dt
    \\
    &\quad +\int_0^s \int_\Omega \Big (\frac{\vphi_{\eps _n}'(\abs{\nabla
      \bar u_n})}{\abs{\nabla \bar u_n}}\nabla \bar u_n-\frac{\vphi_{\eps _n}'(\abs{\nabla
      \tilde u_n})}{\abs{\nabla \tilde u_n}}\nabla \bar u_n \Big
      )\nabla v \,\dv{x}\,\phi_s(t)\, dt
    \\
    &\quad - \int_0^s\int _\Omega \Big (\frac{\vphi_{\eps _m}'(\abs{\nabla
      \bar u_m})}{\abs{\nabla \bar u_m}}\nabla \bar u_m-\frac{\vphi_{\eps _m}'(\abs{\nabla
      \tilde u_m})}{\abs{\nabla \tilde u_m}}\nabla \bar u_m \Big
      )\nabla v \,\dv{x}\,\phi_s(t)\, dt
    \\
    &\quad -\int_0^s  (\gH(\tilde u_n)(t)\bar u_n(t)-\gH(\tilde u_m)(t)\bar u_m(t),
      v)_H \,\phi_s(t)\,dt
    \\
    &\quad + \int_0^s( \bar f_n(t)-\bar f_m(t),v)_H \,\phi_s(t)\,dt
    \\
    &=:  I^{n,m}_1+I^{n,m}_2+I^{n}_3+I^{m}_4+I^{n}_5+I^{m}_6 + I^{n,m}_7+I^{n,m}_8\,.
  \end{align*}
  Since $\phi_s(\cdot)v \in L^\infty(0,T;V) \hookrightarrow
  L^\infty(0,T;L^{p^*}(\Omega))$ and $(p^*)'\le \frac{p^*}{r-1}$ due to
  $r\le p\frac{d+2}{2d}+1$, we deduce from \eqref{konv1} and $\bar
  f_n \to f$ in $L^2(0,T;H)$ that $I_1^{n,m}$, $I_2^{n,m}$,
  $I_7^{n,m}$ and $I_8^{n,m}$ converge
  to zero for $n,m \to \infty$. Using Lemma \ref{lem:eps} we get
  \begin{align}\label{eq:i3}
    \abs{I_3^n} \le c \vphi'(\eps _n) \int_0^s \int_\Omega \abs
    {\nabla v} \phi_s \,\dv{x}\, dt \to 0 \qquad \text {for } n\to \infty.  
  \end{align}
  In the same way we get that $I_4^m$ converges to zero for
  $m \to \infty$. There exists $\ell \in \N$ such that
  $(\ell-1)\tau_n <s\le \ell\tau_n$. Using the definition of $\bar u_u$,
  $\tilde u_n$, Lemma~\ref{lem:est},
  $\max _{t \in (0,T)}|\phi_s(t)| \le 1 $ and  Young's inequality  we get
  \begin{align}
    \begin{split}\label{eq:fe}
      \abs{I_5^n} &\le c\, \tau_n^2 \sum_{k=1}^\ell
      \int_\Omega\frac{\vphi_{\eps_n}'(|\nabla \prescript{\eps _n\hspace{-.25mm}}{}
        {u}^{k-1}_{h_n}|)}{|\nabla \prescript{\eps _n\hspace{-.25mm}}{}
        {u}^{k-1}_{h_n}|} \abs{\nabla d_{\tau_n} \prescript{\eps
          _n}{} {u}^{k}_{h_n}}\,\abs{\nabla v} \dv{x}
      \\
      &\le {\gamma(\eps _n)} \tau_n^2 \sum_{k=1}^\ell
      \int_\Omega\frac{\vphi_{\eps_n}'(|\nabla \prescript{\eps _n}{}
        {u}^{k-1}_{h_n}|)}{|\nabla \prescript{\eps _n}{}
        {u}^{k-1}_{h_n}|} \abs{\nabla d_{\tau_n} \prescript{\eps
          _n}{} {u}^{k}_{h_n}}^2\, \dv{x}
      \\
      &\quad + \frac c{\gamma(\eps _n)} \tau_n^2 \sum_{k=1}^\ell
      \int_\Omega\frac{\vphi_{\eps_n}'(|\nabla \prescript{\eps _n}{}
        {u}^{k-1}_{h_n}|)}{|\nabla \prescript{\eps _n}{}
        {u}^{k-1}_{h_n}|} \abs{\nabla v}^2\, \dv{x}
      \\
      &\le {\gamma(\eps _n)} E_{\vphi_{\eps_n}}[u^0]+ \frac
      {c\,\vphi''(\eps_n)}{\gamma(\eps _n)} \, \tau_n\, \|\nabla
      v\|^2_H
      \\
      &\le {\gamma(\eps _n)} E_{\vphi}[u^0]+ \frac
      {c\,\vphi''(\eps_n)}{\gamma(\eps _n)} \, \tau_n\, \|\nabla
      v\|^2_H,
    \end{split}
  \end{align}
  where we also used
  $\frac {\vphi '_\eps (t)}{t} \le \kappa_0^{-1}\vphi''(\eps)$ due to 
  (C2), (C3), and ${\vphi _\eps (t) \le \vphi(t)}$. Since
  $v \in W^{1,\infty}_0(\Omega)$, the terms in the last line of the
  previous estimate converge to zero since
  ${\tau_n = o(\vphi''(\veps_n)^{-1})}$ as then, e.g.,
  $\gamma^2(\veps_n) = \tau_n \vphi''(\veps_n)$ satisfies
  $\gamma(\veps_n) = o(1)$ and
  $\tau_n \vphi''(\veps_n)/ \gamma(\veps_n)= o(1)$ as $n \to
  \infty$. The term $I_6^m$ is treated analogously. 
  Since $\bigcup _{k\in \N}V_{h_{k}}$ is dense
  in $H$, we have shown that for every $s\in (0,T)$ the sequence
  $(\hat u_n(s))_{n\in \N}$ is a weak Cauchy sequence in $H$. Thus, for every
  $s \in (0,T)$ there exists $w(s) \in H$ such that
  $\hat u_n(s) \rightharpoonup w(s)$ in $H$. From this we deduce as in
  the proof of Theorem~\ref{maintheorem} that 
  for almost every $t \in (0,T)$
  \begin{align}\label{eq:point1}
    \hat u_n(t) \rightharpoonup \bar  u(t) \qquad \text{in } H\,.
  \end{align}
  However we need to prove $\bar u_n(t) \rightharpoonup \bar u(t)$ in
  $H$ for almost all $t \in (0,T)$. To this end we proceed as follows:
  We equip the set $V_{h_n}$, $n \in \N$, with the
  $W^{1,2}_0(\Omega)$-norm and denote this space by $X_n$.  For given
  $m\in \N$ let $n\ge m$ be arbitrary. Then we get, using that
  $\hat u_n(t)-\bar u_n(t)= d_{\tau_n} \prescript{\eps _n\hspace{-.25mm}}{} {u}
  _{h_n}^{k}(t -k\tau_n)$ on $I_{k}^{\tau_n}$
  \begin{align*}
    \|\hat u_n -\bar u_n \|_{L^1(0,T; X_m^\ast)}
    &\le     \|\hat u_n -\bar u_n \|_{L^{1}(0,T; X_n^\ast)}
    \\
    &= \sum_{k=1}^{K_n} \|d_{\tau_n}  \prescript{\eps _n\hspace{-.25mm}}{} {u}^k_{h_n}\|_{X^\ast_n}\int_{I^{\tau
      _n}_k} |t -k\tau_n|\, dt
    \\
    &= \frac {\tau_n}{2}\tau_n \sum_{k=1}^{K_n} \|d_{\tau_n}  \prescript{\eps _n\hspace{-.25mm}}{} {u}^k_{h_n}\|_{X_n^\ast}.
  \end{align*}
  Since $(V,H,V^*)$ and $(W^{1,2}_0(\Omega), H,(W^{1,2}_0(\Omega))^*)$ are Gelfand triples we get
  $\big( d_{\tau_n} \prescript{\eps _n\hspace{-.25mm}}{} {u}^k_{h_n},v \big
  )_H=\langle d_{\tau_n} \prescript{\eps _n\hspace{-.25mm}}{} {u}^k_{h_n},v \rangle
  _V = \langle d_{\tau_n} \prescript{\eps _n\hspace{-.25mm}}{} {u}^k_{h_n},v \rangle
  _{W^{1,2}_0(\Omega)}$ for $v \in W^{1,\infty}_0(\Omega)$. This and 
  \eqref{eq:phi-rg-new} yields 
  \begin{align*}
    \|d_{\tau_n}  \prescript{\eps _n\hspace{-.25mm}}{} {u}^k_{h_n}\|_{X_n^\ast}
    &= \sup_{v \in X_n\atop \|v\|_{W^{1,2}_0(\Omega)}
      \le 1}  \big (  d_{\tau_n}  \prescript{\eps _n\hspace{-.25mm}}{}
      {u}^k_{h_n},v\big ) 
    \\
    &= \sup_{v \in X_n\atop \|v\|_{W^{1,2}_0(\Omega)}
      \le 1} \Big [ - \big (  A_{\eps_n} (\nabla \prescript{\eps _n\hspace{-.25mm}}{} {u}_{h_n}^k), \nabla v \big )
      - (\h(\prescript{\eps _n\hspace{-.25mm}}{} {u}_{h_n}^{k-1})\, \prescript{\eps _n\hspace{-.25mm}}{} {u}_{h_n}^k,v)
    \\
    &\hspace{27mm}+ (f_h(t_k),v) + \big (\Fe_h^k, \nabla v  \big ) \Big ] .
  \end{align*}
  Using H\"older's inequality,
  $\frac {\vphi '_\eps (t)}{t} \le \kappa_0^{-1}\vphi''(\eps)$, the
  properties of $\vphi$ and Young's inequality we obtain
  \begin{align*}
    \bigabs{(  A_{\eps_n} (\nabla \prescript{\eps _n\hspace{-.25mm}}{} {u}_{h_n}^k), \nabla v )}
    & \le\Big ( \int_\Omega\frac{\big (\vphi_{\eps_n}'(|\nabla \prescript{\eps _n\hspace{-.25mm}}{}
        {u}^{k}_{h_n}|)\big )^2}{|\nabla \prescript{\eps _n\hspace{-.25mm}}{}
        {u}^{k}_{h_n}|^2} \abs{\nabla  \prescript{\eps _n\hspace{-.25mm}}{} {u}^{k}_{h_n}}^2 \dv{x}
      \Big )^{\frac 12} \|\nabla v\|_H
    \\
    & \le c\, \vphi''(\eps_n) \, \|\nabla v\|_H^2+  c \int_\Omega\vphi_{\eps_n}(|\nabla \prescript{\eps _n\hspace{-.25mm}}{}
        {u}^{k}_{h_n}|) \dv{x}  \,.
  \end{align*}
  Similarly as in \eqref{eq:gron} we get
  \begin{align*}
     \abs{(\h(\prescript{\eps _n\hspace{-.25mm}}{} {u}_{h_n}^{k-1})\, \prescript{\eps   _n\hspace{-.25mm}}{} {u}_{h_n}^k,v)}
    &\le c\big (1 +\|v\|_H^2 +\|\prescript{\eps _n\hspace{-.25mm}}{}
      {u}_{h_n}^{k-1}\|_V^p +\|\prescript{\eps _n\hspace{-.25mm}}{}
      {u}_{h_n}^{k}\|_V^p  \big )\,.
  \end{align*}
  From Assumption \ref{ass:data1} we conclude
  \begin{align*}
    \abs{(f_h(t_k),v) } \le \|v\|_H^2 + \|f_h(t_k)\|^2_H\,.
  \end{align*}
  Using Lemma~\ref{lem:est},  Young's inequality and $\frac {\vphi
    '_\eps (t)}{t} \le \kappa_0^{-1}\vphi''(\eps)$  we get
  \begin{align*}
    \bigabs{ (\Fe_h^k, \nabla v  \big ) }
    &\le c\, \tau_n \int_\Omega\frac{\vphi_{\eps_n}'(|\nabla \prescript{\eps _n\hspace{-.25mm}}{}
      {u}^{k-1}_{h_n}|)}{|\nabla \prescript{\eps _n\hspace{-.25mm}}{}
      {u}^{k-1}_{h_n}|} \abs{\nabla d_{\tau_n} \prescript{\eps _n\hspace{-.25mm}}{} {u}^{k}_{h_n}}\,\abs{\nabla v} \dv{x}
    \\
    &\le {\gamma(\eps _n)} \tau_n \int_\Omega\frac{\vphi_{\eps_n}'(|\nabla \prescript{\eps _n\hspace{-.25mm}}{}
      {u}^{k-1}_{h_n}|)}{|\nabla \prescript{\eps _n\hspace{-.25mm}}{}
      {u}^{k-1}_{h_n}|} \abs{\nabla d_{\tau_n} \prescript{\eps _n\hspace{-.25mm}}{} {u}^{k}_{h_n}}^2\, \dv{x}
    \\
    &\quad + \frac c{\gamma(\eps _n)} \tau_n
      \int_\Omega\frac{\vphi_{\eps_n}'(|\nabla \prescript{\eps _n\hspace{-.25mm}}{}
      {u}^{k-1}_{h_n}|)}{|\nabla \prescript{\eps _n\hspace{-.25mm}}{}
      {u}^{k-1}_{h_n}|} \abs{\nabla v}^2\, \dv{x}
    \\
    &\le {\gamma(\eps _n)} \tau_n\! \int_\Omega\!\frac{\vphi_{\eps_n}'(|\nabla \prescript{\eps _n\hspace{-.25mm}}{}
      {u}^{k-1}_{h_n}|)}{|\nabla \prescript{\eps _n\hspace{-.25mm}}{}
      {u}^{k-1}_{h_n}|} \abs{\nabla d_{\tau_n} \prescript{\eps _n\hspace{-.25mm}}{} {u}^{k}_{h_n}}^2\, \dv{x}
      +\frac {c\,\vphi''(\eps_n)}{\gamma(\eps _n)}  \tau_n \|\nabla
      v\|^2_H.
  \end{align*}
  Consequently, we proved
  \begin{align*}
    \|d_{\tau_n}  \prescript{\eps _n\hspace{-.25mm}}{} {u}^k_{h_n}\|_{X_n^\ast}
    & \le {\gamma(\eps _n)} \tau_n\! \int_\Omega\!\frac{\vphi_{\eps_n}'(|\nabla \prescript{\eps _n\hspace{-.25mm}}{}
      {u}^{k-1}_{h_n}|)}{|\nabla \prescript{\eps _n\hspace{-.25mm}}{}
      {u}^{k-1}_{h_n}|} \abs{\nabla d_{\tau_n} \prescript{\eps_n\hspace{-.25mm}}{} {u}^{k}_{h_n}}^2\, \dv{x}
      +\frac {c\,\vphi''(\eps_n)}{\gamma(\eps _n)}  \tau_n
    \\
    &\quad +c\, \vphi''(\eps_n) + c \int_\Omega\vphi_{\eps_n}(|\nabla \prescript{\eps _n\hspace{-.25mm}}{}
      {u}^{k}_{h_n}|) \dv{x}  +c\, \|\prescript{\eps _n\hspace{-.25mm}}{}
      {u}_{h_n}^{k-1}\|_V^p
    \\
    &\quad +c\,\|\prescript{\eps _n\hspace{-.25mm}}{} {u}_{h_n}^{k}\|_V^p +\|f_h(t_k)\|^2_H +c
  \end{align*}
  and thus 
  \begin{align*}
    &\|\hat u_n -\bar u_n \|_{L^1(0,T; X_m^\ast)}
    \\
    &\le c  {\tau_n}\,\tau_n \sum_{k=1}^{K_n}  \int_\Omega\vphi_{\eps_n}(|\nabla \prescript{\eps _n\hspace{-.25mm}}{}
      {u}^{k}_{h_n}|) \dv{x}  + c {\tau_n}\,\tau_n
      \sum_{k=1}^{K_n}\vphi''(\eps_n)     + c {\tau_n}\,\tau_n \sum_{k=1}^{K_n} \frac {\vphi''(\eps_n)}{\gamma(\eps _n)}  \tau_n 
    \\
    &\quad + c {\tau_n}\,\gamma(\eps_n)\,\tau_n^2 \sum_{k=1}^{K_n}\int_\Omega\!\frac{\vphi_{\eps_n}'(|\nabla \prescript{\eps _n\hspace{-.25mm}}{}
      {u}^{k-1}_{h_n}|)}{|\nabla \prescript{\eps _n\hspace{-.25mm}}{}
      {u}^{k-1}_{h_n}|} \abs{\nabla d_{\tau_n} \prescript{\eps
      _n}{} {u}^{k}_{h_n}}^2\, \dv{x} + c {\tau_n}\,\tau_n
      \sum_{k=1}^{K_n}1
    \\
    &\quad +  c {\tau_n}\,\tau_n
      \sum_{k=1}^{K_n} \|\prescript{\eps _n\hspace{-.25mm}}{}  {u}_{h_n}^{k-1}\|_V^p
      +  c {\tau_n}\,\tau_n
      \sum_{k=1}^{K_n}\|\prescript{\eps _n\hspace{-.25mm}}{}  {u}_{h_n}^{k}\|_V^p  +  {\tau_n}\,\tau_n
      \sum_{k=1}^{K_n} \|f_h(t_k)\|^2_H  \,.
  \end{align*}
  Using $\tau_n =o(\vphi''(\eps_n))$, the estimates
  \eqref{eq:ener_bound1} and  \eqref{eq:apri2}  as well as 
  Assumption \ref{ass:data1} we see that all terms on
  the right-hand side converge to zero for $n\to \infty$. A
  diagonal procedure implies for all $m\in \N$ and almost all
  $t \in (0,T)$ 
  \begin{align*}
    \hat u_n(t) -\bar  u_n(t) \to 0 \qquad \text{in } X_m^\ast\,,
  \end{align*}
  which together with \eqref{eq:point1}, the properties of the Gelfand
  triple with the spaces $W^{1,2}_0(\Omega)$, $L^2(\Omega)$, and $(W^{1,2}_0(\Omega))^*$  and
  the density of $\bigcup _{k\in \N}X_{k}$ in $H$ yields 
  \begin{align}
    \label{eq:point3}
    \bar u_n(t) \rightharpoonup \bar  u(t) \qquad \text{in } H\,.
  \end{align}  
  This and \eqref{eq:u=u} implies 
    \begin{align}
    \label{eq:point4}
    \tilde u_n(t) \rightharpoonup \bar  u(t) \qquad \text{in } H\,.
  \end{align}  

  {\it (iv) verification of condition \eqref{Hiranovoraussetzung}$_4$}:
  We first show that $\gH^*= \gH(\bar u)\bar u 
  $ in $L^\infty (0,T;L^{\frac{p^*}{r-1}}(\Omega))$. In view of \eqref{eq:apri2},
  \eqref{eq:point3} and \eqref{eq:point4} Proposition \ref{pro:comp}
  yields for all $s \in [1,\infty)$, $q\in [1,p^*)$
  \begin{align}
    \label{eq:conv-s}
    \bar u_n ,\tilde u_n \to \bar u \qquad \text { in } L^s(0,T;L^q(\Omega))\,.
  \end{align}
  Condition (H2) and the theory of Nemyckii operators yields
  (cf.~Lemma~\ref{lem:op}) that
  $\gH\colon L^q(0,T;L^q(\Omega)) \to L^{\frac q{r-2}}(0,T;L^{\frac
    q{r-2}}(\Omega))$, $q\ge \max \{1,r-2\}$, is bounded and
  continuous. This and \eqref{eq:conv-s} yields for all $q\in [\max \{1,r-2\},p^*)$
  \begin{align}
    \label{eq:conv-s1x}
    \gH(\bar u_n) , \gH(\tilde u_n) \to \gH(\bar u) \qquad \text { in } L^{\frac q{r-2}}(0,T;L^{\frac
    q{r-2}}(\Omega)) \,.
  \end{align}
  From this and \eqref{eq:conv-s} follows for all $q\in [\max \{1,r-1\},p^*)$
  \begin{align}
    \label{eq:conv-s1}
    \gH(\bar u_n)\bar u_n , \gH(\tilde u_n)\bar u_n \to \gH(\bar u)\bar u \qquad \text { in } L^{\frac q{r-1}}(0,T;L^{\frac
    q{r-1}}(\Omega)) \,,
  \end{align}
  which together with \eqref{konv1} proves $\gH^*= \gH(\bar u)\bar u
  $ in $L^\infty (0,T;L^{\frac{p^*}{r-1}}(\Omega))$.
  Using \eqref{eq:ga1}, the integration by parts formula we obtain for
  all $\phi \in C_0^\infty (\R)$ and all $v\in X_{h_m}$, where
  $n \ge m$
  \begin{align*}
    &(\hat u_n(T),v)_H\phi(T)-(\hat u_n(0),v)_H\phi(0)\\
    &=\int_0^T(\hat u_n(t),v)_H\,\phi'(t) 
      -\big (\langle A\bar u_n(t),v\rangle_V- \langle
       E_n(t),v\rangle_V- \langle
      F_n(t),v\rangle_V\big )\phi(t) \,dt
    \\
    &\quad +\int_0^T \big ((\bar f_n(t),v)_H - (\gH(\tilde u_n)(t)\bar
      u_n(t),v)_H\big )\phi(t)\,\dv{t}\,. 
  \end{align*}
  Notice that the last two terms in the first line of the right-hand
  side converge to zero by similar arguments as in \eqref{eq:i3} and
  \eqref{eq:fe}. Further we have $\phi(\cdot)v \in L^\infty(0,T;V) \hookrightarrow
  L^\infty(0,T;L^{p^*}(\Omega))$ and $(p^*)'< \frac{p^*}{r-1}$, which
  holds due to $r\le p\frac{d+2}{2d}+1$ and $p>\frac{2d}{d+2}$.
  Thus, the convergences in \eqref{konv} and
  \eqref{eq:conv-s1}, the convergence $\bar f_n \to f $ in
  $L^{p'}(0,T;H)$, the identity of sets $X_{h_k}= V_{h_k}$, the
  density of $\bigcup _{k\in \N}V_{h_k}$ in $V$ and $H$, and $\bar u =\hat u$
  in $L^2(0,T;H)$ yield
  \begin{align*}
    &( u^\ast,v)_H\phi(T)-(u^0,v)_H\phi(0)\\
    &=\!\int_0^T\!\!(\bar u(t),v)_H\,\phi'(t) 
      +\big ( (f(t),v)_H-\langle A^\ast(t),v\rangle_V -(\gH(\bar u(t))\bar
      u (t),v)_H \big )\,\phi(t)      \,dt
  \end{align*}
  for all $\phi \in C_0^\infty (\R)$ and all $v\in V$. For $\phi
  \in C_0^\infty (0,T)$ this and the definition of the generalized
  time derivative together with $H\hookrightarrow V^*$ imply
  \begin{align}
    \label{eq:ut1}
    \frac{d\bar u }{dt} =  f - A^\ast-\gH(\bar u)\bar
      u \qquad \text{in } L^{p'}(0,T; V^\ast). 
  \end{align}
  Moreover, by standard arguments we get $\bar u \in C(\bar I; H)$,
  $u^\ast =\bar u(T)$, and
  $\hat u_n(T) =\bar u_n(T) \rightharpoonup \bar u(T)$ in $H$. Using
  \eqref{eq:ga1} for $v=\bar u_n(t)$ and
  \begin{align*}
      \Big \langle \frac {d \hat u_{n}  }{dt},\bar u_n\Big \rangle
    _{L^p(0,T;V)} = \tau _n \sum _{k=1}^{K_n} (d_\tau  u_{M_n}^k,
    u_{M_n}^k)_H \ge \frac 12 \|\bar u _n(T)\|_H^2 -
    \frac 12 \|u _n^0\|_H^2 
  \end{align*}
  we obtain with
  $\langle G_n(t),v\rangle _V:=\big ((\gH(\bar u)(t)\bar u (t)-(\gH(\tilde
  u)(t)\bar u (t),v\big )_H$
  \begin{align*}
    &\langle A \bar u_{n} +\gH(\bar u_n)\bar u_n  ,\bar u_n\rangle _{L^p(0,T;V)}
    \\
    &= \langle \bar f_{n}, \bar u_n\rangle_{L^p(0,T;H)}+\langle
      E_{n}+F_n+G_n, \bar u_n\rangle_{L^p(0,T;V)}- \Big \langle \frac {d \hat u_{n}  }{dt},\bar u_n\Big \rangle
    _{L^p(0,T;V)}  
    \\
    &\le \langle \bar f_{n}, \bar u_n\rangle_{L^p(0,T;H)}+\langle
      E_{n}+F_n+G_n, \bar u_n\rangle_{L^p(0,T;V)} + \frac 12
      \|u _n^0\|_H^2 - \frac 12 \|\bar u _n(T)\|_H^2 .
  \end{align*}
  Similarly as in \eqref{eq:i3} and \eqref{eq:fe} we obtain
  \begin{align*}
    &\abs{\langle E_n, \bar u_n\rangle _{L^p(0,T;V)}}
    \le c\vphi  '(\eps _n)\int_0^T\int_\Omega \abs{\nabla \bar u_n}
      \dv{x}\, dt \to 0 \qquad n\to \infty,
    \\
    &\abs{\langle F_n, \bar u_n\rangle _{L^p(0,T;V)}}
      \\&\le {\gamma(\eps _n)} \tau_n^2 \sum_{k=1}^\ell
      \int_\Omega\frac{\vphi_{\eps_n}'(|\nabla \prescript{\eps _n\hspace{-.25mm}}{}
        {u}^{k-1}_{h_n}|)}{|\nabla \prescript{\eps _n\hspace{-.25mm}}{}
        {u}^{k-1}_{h_n}|} \abs{\nabla d_{\tau_n} \prescript{\eps
          _n\hspace{-.25mm}}{} {u}^{k}_{h_n}}^2\, \dv{x}
      \\
      &\quad + \frac c{\gamma(\eps _n)} \tau_n^2 \sum_{k=1}^\ell
      \int_\Omega\frac{\vphi_{\eps_n}'(|\nabla \prescript{\eps _n\hspace{-.25mm}}{}
        {u}^{k-1}_{h_n}|)}{|\nabla \prescript{\eps _n\hspace{-.25mm}}{}
        {u}^{k-1}_{h_n}|} \abs{\nabla \prescript{\eps _n\hspace{-.25mm}}{}
        {u}^{k}_{h_n}}^2\, \dv{x}
      \\
      &\le {\gamma(\eps _n)}\Big ( E_{\vphi}[u^0]  +
        \|f\|_{L^2(0,t;H)}^2+(\eps_n^p +\delta^p_n)T\abs{\Omega} +o(1)\Big )
    \\
    &\quad + \frac
      {c\,\tau_n}{\gamma(\eps _n)} \Big (E_{\vphi}[u^0] +\|u^0\|_H
        +(\eps_n^p +\delta^p_n)T\abs{\Omega} +o(1)\Big)
        \to 0 \quad n\to \infty,
  \end{align*}
  where we used that $L^p(0,T;V) $ embeds into $L^1(0,T;V)$; the
  properties of $\vphi$, \eqref{eq:apriori}, \eqref{eq:ener_bound1}, 
  the choice $\tau_n =o(\vphi''(\eps_n)^{-1})$ and Assumption~\ref{ass:data1}. In view of
  \eqref{eq:conv-s1} and \eqref{konv1}  we get 
  $\abs{\langle G_n, \bar u_n\rangle _{L^p(0,T;V)}}\to 0$. Thus
  \eqref{konv1}, $\bar f_n \to f $ in
  $L^{p'}(0,T;H)$ and the lower weak semicontinuity of the norm imply
  \begin{align*}
      \limsup _{n\to \infty} \langle A \bar u_{n} (t) +\gH(\bar u_n)\bar u_n  ,\bar u_n\rangle _{L^p(0,T;V)} 
    &\le \langle f, \bar u\rangle_{L^p(0,T;H)} +
      \frac 12
      \|u ^0\|_H^2 - \frac 12 \|\bar u (T)\|_H^2.
  \end{align*}
  From \eqref{eq:ut1}, the integration by parts formula and
  \eqref{konv1}, \eqref{eq:conv-s1} we get
  \begin{align*}
     \langle f, \bar u\rangle_{L^p(0,T;H)}
    = \frac 12 \|\bar u
    (T)\|_H^2 -\frac 12 \|u ^0\|_H^2 + \lim_{n \to \infty}
    \langle A \bar u_{n}+\gH(\bar u_n)\bar u_n  ,\bar u\rangle _{L^p(0,T;V)} \,. 
  \end{align*}
  The last two inequalities imply that also condition
  \eqref{Hiranovoraussetzung}$_4$ is satisfied.

  Thus, we have verified all conditions in \eqref{Hiranovoraussetzung}
  and consequently Proposition~\ref{Hirano} together with
  \eqref{konv1} implies $A^\ast +H^\ast= A\bar u +\gH(\bar u)\bar u$
  in $L^{p'}(0,T;V^\ast)$. This and \eqref{eq:ut1} yield
  \begin{align*}
    \frac{d\bar u }{dt}  + A\bar u +\gH(\bar u)\bar u= f  \qquad \text{in } L^{p'}(0,T;
    V^\ast),
  \end{align*}
  i.e.~$\bar u $ is a solution of \eqref{eq:mp}.  
\end{proof}

\begin{remark}
  For $p=2$ we have to distinguish between the cases $d=2$ and $d\ge
  3$. In the latter one Theorem~\ref{maintheorem1} holds as stated and
  also the proof is the same. If $d =2$ the embedding
  $W^{1,2}_0(\Omega) \hookrightarrow L^{s}(\Omega)$, $s\in [1,\infty)$
  is different from the other cases we considered. Thus, estimate
  \eqref{eq:gron} has to be adapted and results in the restriction
  $r<3$. Consequently, in Theorem~\ref{maintheorem1} we have to
  require $r \in (2,3)$ if $p=2$ and $d=2$. 
\end{remark}

\bibliographystyle{amsalpha}
\bibliography{bib_project}

\end{document}

%% file: sb_macros-mod.tex
\def\R{\mathbb{R}}

\def\N{\mathbb{N}}

\def\cD{\mathcal{D}}

\def\a{\alpha}

\def\d{\delta}

\def\veps{\varepsilon}

\def\vphi{\varphi}
\def\O{\Omega}

\newcommand{\dv}[1]{\,{\mathrm d}#1}

\newcommand{\wcheck}[1]{#1\hspace{-.8ex}\mbox{\huge {\lower.45ex \hbox{$\textstyle \check{}$}}} \hspace{.5ex}}

\DeclareMathOperator{\diver}{div}

\let\oldmarginpar\marginpar
\renewcommand\marginpar[1]{
  \oldmarginpar[\raggedleft\footnotesize #1]
  {\raggedright\footnotesize #1}}